\documentclass[ASNA]{USG}

\usepackage{bm}
\usepackage{color}
\usepackage{amsfonts}
\usepackage{amsmath}
\usepackage{amssymb}
\usepackage{graphicx}
\usepackage{listings} 
\usepackage{multirow}
\usepackage{enumitem}
\usepackage{hyperref}
\usepackage{lineno}
\usepackage{indentfirst}
\usepackage{algorithm}
\usepackage{algpseudocode}
\usepackage{booktabs}
\usepackage{appendix}
\usepackage{comment}

\newcommand{\ignore}[1]{}

\begin{document}

\title{A Linear Variable-Step Embedded ETD Scheme with Uniform-in-Time Stability for the 2D Navier--Stokes Equations}

\author[1]{Haifeng Wang}
\author[1]{Xiaoming Wang}
\author[2,1]{Min Zhang}

\corres{Xiaoming Wang (\email{wxm@eitech.edu.cn})}

\address[1]{\orgdiv{School of Mathematical Sciences}, \orgname{Eastern Institute of Technology}, \orgaddress{\state{Zhejiang}, \city{Ningbo}, \country{China}}}

\address[2]{\orgdiv{School of Mathematical Sciences}, \orgname{Shanghai Jiao Tong University}, \orgaddress{\city{Shanghai}, \country{China}}}

\newtheorem{exm}{Example}

\abstract[ABSTRACT]{
We propose a linear variable-step exponential time-differencing method for the incompressible Navier--Stokes equations in vorticity--streamfunction formulation on a two-dimensional periodic box. The method consists of a second-order scheme and an embedded first-order variant, yielding a natural mechanism for adaptive time stepping and a posteriori error control. Each time step requires only uniquely solvable linear problems: two heat equation solves, efficiently handled by Fourier methods in the periodic setting, and one linear scalar auxiliary-variable equation, evaluated via Laplace transform and Talbot's numerical inverse transform.

The construction combines the ETD framework, a mean-reverting scalar auxiliary variable (mr-SAV), and second-order extrapolation of the nonlinear term. The mean-reverting correction enables long-time stability while preserving full linearity, distinguishing the method from related mr-SAV schemes that require nonlinear algebraic solves. We prove unconditional long-time stability: for uniformly bounded $L^2$ forcing, the discrete vorticity remains bounded in $L^\infty(0,\infty;L^2)$ for all Reynolds numbers and time-step sizes. Numerical experiments demonstrate second-order accuracy, long-time stability, and effective adaptive error control.
}

%\MSC[2020] 65M12 \sep 65L04 \sep 65L05 \sep 76D05 

%\end{frontmatter}

% \linenumbers


%\linenumbers

\maketitle
\section{Introduction}

The incompressible Navier--Stokes equations (NSE) are a central model in fluid dynamics and continue to serve as a fundamental testbed for the design and analysis of time-marching methods for nonlinear dissipative PDEs. In this work, we consider the two-dimensional NSE in vorticity--streamfunction formulation on a periodic box,
\begin{equation}\label{eq:nse-vorticity-intro}
\begin{aligned}
\omega_t + \bm u \cdot \nabla \omega &= \nu \Delta \omega + f,\\
\Delta \psi &= \omega, \qquad \bm u = \nabla^\perp \psi := (\partial_y \psi,-\partial_x \psi),
\end{aligned}
\end{equation}
supplemented with periodic boundary conditions and the usual zero-mean conditions on the vorticity $\omega$, streamfunction $\psi$, velocity $\bm u$, and forcing $f$. {Here $\nu>0$ is the kinematic viscosity.

For long-time simulation, a numerical method should capture not only short-time accuracy but also the dissipative structure of the underlying dynamics. When the forcing is uniformly bounded in time, the continuous NSE admits a uniform-in-time bound in the natural energy/enstrophy space; this property lies behind the existence of absorbing sets, attractors, and long-time statistical states, and is therefore indispensable in the numerical approximation of asymptotic dynamics and time-averaged observables \cite{constantin1988navier,foias2001navier,majda2006nonlinear,temam2012infinite}. A time discretization that fails to inherit such control may produce spurious growth or distorted long-time statistics even when it performs satisfactorily over moderate time intervals.

At the same time, practical long-time computations strongly benefit from time adaptivity. Solutions of the NSE often alternate between relatively smooth epochs and episodes of rapid variation, so fixed-step methods can be unnecessarily expensive. This makes variable-step and, more generally, variable-order/variable-step strategies attractive from a computational point of view. For fluid problems, adaptive multistep and embedded methods of variable-step or variable-step/variable-order (VOVS/VSVO) type have been studied, for instance, by DeCaria and collaborators and by Layton and collaborators; see, e.g., \cite{li2020adaptive,decaria2021embedded,decaria2021variable}. These works demonstrate the practical value of adaptive low-complexity time integrators for incompressible flow, especially in conjunction with IMEX or BDF-type frameworks.

The construction of adaptive methods that are both efficient and provably stable for the NSE remains challenging, however. Fully implicit schemes enjoy strong stability properties but typically require nonlinear solves or the repeated solution of nontrivial linear systems at each step. IMEX and linearly implicit methods reduce this cost by treating the nonlinearity explicitly, but their long-time stability theory under variable step sizes is much more delicate. Exponential time-differencing (ETD) methods are particularly attractive in periodic settings, where the linear viscous part can be handled very efficiently in Fourier space. They have been widely used for stiff semilinear problems and have also been explored for fluid equations; nevertheless, rigorous long-time stability results for higher-order ETD discretizations of the forced NSE remain scarce.

A recent line of work has shown that scalar auxiliary variable ideas can be adapted beyond gradient flows to help stabilize time-marching methods for the NSE \cite{huang2022new,yang2021new,li2020sav}. Moreover, mean-reverting scalar auxiliary variable (mr-SAV) formulations introduce an additional scalar variable designed to encode a dissipative correction compatible with the forced long-time dynamics. In our earlier work \cite{WangWangZhang2026}, we combined this idea with ETD to construct a variable-step embedded framework with favorable long-time stability properties. That approach, however, involves the solution of a nonlinear algebraic equation for the scalar auxiliary variable at each time step. While this nonlinear solve is scalar and can be carried out efficiently, it is still desirable---both conceptually and computationally---to develop a fully linear alternative.

The purpose of the present paper is to provide such a scheme. We develop a \emph{linear} variable-step embedded ETD method for the two-dimensional NSE in vorticity--streamfunction form. The method consists of a second-order scheme together with an embedded first-order companion, thereby providing a natural mechanism for adaptive time stepping and a posteriori error estimation. The key design combines ETD for the viscous part, explicit multistep treatment of the nonlinear advection, and a new mr-SAV formulation arranged so that each time step requires only linear subproblems. More precisely, in the periodic setting the update reduces to the evaluation of two time-dependent heat equations, together with a linear equation for the scalar auxiliary variable, which is treated through Laplace transform and Talbot's algorithm for numerical inverse
Laplace transform. Thus, in contrast to \cite{WangWangZhang2026}, the present method avoids any nonlinear algebraic solve at the time-discrete level.

Our main analytical result shows that this simplification does not come at the expense of long-time stability. Roughly speaking, we prove that if the external forcing is uniformly bounded in $L^2$, then the discrete vorticity generated by the second-order scheme remains uniformly bounded in $L^\infty(0,\infty;L^2)$, for arbitrary Reynolds numbers and without any restriction on the time-step sizes. A precise theorem statement is given in Section~\ref{sec:4}. This establishes a long-time stability property for a variable-step ETD scheme of genuinely higher order in the presence of nontrivial forcing.

In addition to the main embedded ETD method, we also report a preliminary variable-step/variable-order (VSVO) experiment based on combining the adaptive framework with a second-order variable-step BDF2 discretization for the NSE. This VSVO variant is included as a secondary computational exploration rather than a main theoretical contribution. The numerical tests indicate that this alternative adaptive strategy can also be used successfully in some regimes.
%, although the most robust performance in the present study is delivered by the new linear embedded ETD--mr-SAV scheme.

The contributions of this paper can be summarized as follows.
\begin{itemize}
    \item We construct a new fully linear and uniquely solvable embedded first--second order variable-step ETD scheme for the forced two-dimensional NSE in vorticity--streamfunction formulation.
    \item We show that the second-order member of the pair is long-time stable in the sense of a uniform-in-time $L^2$ bound for the discrete vorticity, under uniformly bounded forcing, for arbitrary Reynolds numbers and arbitrary step sizes.
    \item We design the method so that only linear subproblems are solved at each time step, thereby distinguishing the present approach from our earlier ETD--mr-SAV scheme in \cite{WangWangZhang2026}, which requires a nonlinear algebraic solve for the auxiliary variable.
    \item We demonstrate through numerical experiments that the embedded pair provides effective adaptive error control and captures the expected second-order temporal accuracy and long-time stability.
    \item We also test a VSVO strategy, providing an additional adaptive option of possible independent interest.
\end{itemize}

The paper is organized as follows. In Section~\ref{sec:2}, we introduce the functional setting and the mr-SAV reformulation. Section~\ref{sec:3} presents the linear embedded ETD scheme and its implementation. Section~\ref{sec:4} contains the long-time stability analysis and the statement of the main theorem. Section~\ref{sec:5} reports numerical experiments on accuracy, adaptivity, and long-time behavior, including the auxiliary VSVO tests. Concluding remarks are given in Section~\ref{sec:6}.

\section{Preliminaries and the extended system with mean-reverting SAV reformulation}\label{sec:2}

In this section, we introduce a continuous extended system for the incompressible Navier--Stokes equations based on a mean-reverting scalar auxiliary variable. When the auxiliary variable is initialized at its equilibrium value, the extended system reduces exactly to the original NSE. This reformulation will serve as the starting point for the numerical developments in the next section. Throughout the paper, we denote by $\langle \cdot,\cdot\rangle$ the duality pairing between $\dot H^1_{\mathrm{per}}$ and $H^{-1}_{\mathrm{per}}$ induced by the $L^2$ inner product, and by $\|\cdot\|$ the norm in $L^2(\Omega)$.

We first rewrite the Navier--Stokes equations \eqref{eq:nse-vorticity-intro} in the form
\begin{equation}\label{eqn:vs_ns}
\begin{aligned}
    \frac{\partial \omega}{\partial t} + \nu \mathcal L \omega + B(\omega,\omega) &= f,\\
    \Delta \psi = \omega,\qquad \bm u &= \nabla^\perp \psi := (\partial_y\psi,-\partial_x\psi),
\end{aligned}
\end{equation}
where
\[
\mathcal L\omega = -\Delta \omega, \qquad B(\omega,\omega) = \nabla^{\bot}(\Delta^{-1}\omega)\cdot \nabla \omega=\bm u\cdot\nabla\omega .
\]
In the vorticity--streamfunction formulation, the velocity field is automatically divergence free. Accordingly, we work in the mean-zero periodic Sobolev spaces
\[
H=\Bigl\{\omega\in \dot H^0_{\mathrm{per}}(\Omega): \int_\Omega \omega\,dx=0\Bigr\},\qquad
V=\Bigl\{\omega\in \dot H^1_{\mathrm{per}}(\Omega): \int_\Omega \omega\,dx=0\Bigr\}.
\]

The present analysis is carried out in the vorticity--streamfunction setting, which is particularly convenient for periodic domains and for the enstrophy-based estimates used below. Analogous ideas can also be formulated for the primitive-variable NSE, although some of the sharper estimates available here rely on the special structure of the vorticity equation. For this reason, we restrict attention to \eqref{eqn:vs_ns} throughout the paper.

To prepare for the construction of long-time stable time discretizations, we introduce a scalar auxiliary variable \(q(t)\) and consider the extended system
\begin{subequations}\label{eqn:mr_sav}
\begin{align}
    \frac{\partial \omega}{\partial t} + \nu \mathcal L\omega + q\,B(\omega,\omega) &= f, \label{eqn:mr_sav_1}\\
    \frac{dq}{dt} + \gamma q &= \langle B(\omega,\omega),\omega\rangle + \gamma. \label{eqn:mr_sav_2}
\end{align}
\end{subequations}
Here \(\gamma\ge 0\) is a user-specified mean-reverting parameter. A basic observation is that \(q=1\) is the equilibrium value of the scalar auxiliary variable. Indeed, if \(q(0)=1\), then by the enstrophy conservation property of the nonlinear advection term,
\[
\langle B(\omega,\omega),\omega\rangle = 0,
\]
equation \eqref{eqn:mr_sav_2} yields
\[
\frac{dq}{dt}+\gamma(q-1)=0,
\]
and hence
\[
q(t)\equiv 1 \qquad \text{for all } t\ge 0.
\]
In this case, \eqref{eqn:mr_sav} reduces exactly to the original system \eqref{eqn:vs_ns}. Thus, the extended formulation is consistent with the NSE at the continuous level.

More generally, for arbitrary initial data \(q(0)=q_0\), one has
\[
|q(t)-1|
\le |q_0-1|e^{-\gamma t}
+\frac{1-e^{-\gamma t}}{\gamma}
\bigl\|\langle B(\omega,\omega),\omega\rangle\bigr\|_{L^\infty(0,t)}
\]
for \(\gamma>0\). In particular, the damping term drives \(q(t)\) toward its equilibrium value \(1\), with deviations controlled by the nonlinear inner product and the mean-reverting parameter $\gamma$. At the continuous level, the identity
\[
\langle B(\omega,\omega),\omega\rangle=0
\]
implies that \(q(t)\) relaxes exactly to \(1\). This mean-reverting mechanism is the key structural feature of the reformulation.

The role of \(q(t)\) is to facilitate stable treatments of the nonlinear advection term while preserving consistency with the original NSE. In particular, although the quantity \(\langle B(\omega,\omega),\omega\rangle\) vanishes at the continuous level, its discrete counterpart need not do so exactly. The parameter \(\gamma>0\) introduces a restoring mechanism that suppresses the accumulation of such deviations over long times by driving the auxiliary variable back toward \(1\). This feature is essential for the long-time stability mechanism exploited later in the paper.

When \(\gamma=0\), system \eqref{eqn:mr_sav} reduces to the continuous counterpart of a standard SAV/ZEC-type reformulation; see, for example, \cite{yang2021new,yang2021novel,zhang2024unified,li2022new}. In the present work, however, we focus on the case \(\gamma>0\), where the dissipative mean-reverting term provides the additional control needed for robust long-time behavior. For this reason, following \cite{han2025highly,coleman2024efficient,WangWangZhang2026}, we refer to \eqref{eqn:mr_sav} as the \emph{mean-reverting scalar auxiliary variable} (mr-SAV) extended system.

We conclude this section by emphasizing one point relevant to the present paper. In our earlier ETD--mr-SAV work \cite{WangWangZhang2026}, the stabilization mechanism was coupled with a nonlinear correction to the vorticity field. By contrast, the reformulation introduced here will be used in the next section to derive a different fully linear and uniquely solvable time-discrete treatment, while retaining the same continuous mean-reverting structure.

\section{Exponential time difference mr-SAV schemes}\label{sec:3}
Building upon the mr-SAV extended systems presented in \eqref{eqn:mr_sav}, we proceed to develop two exponential time difference (ETD) time stepping schemes for numerical solution. 

Given an arbitrary terminal time $T$ and a set of non-overlapping time nodes $0 = t_0 < t_1 < \cdots < t_N = T$ with the $k$th time step size $\tau_k = t_k -t_{k-1}$, we specify that the time partition here can be non-uniform. Let $\Psi^{n}$ denote the numerical approximation of $\Psi(t)$ at the time $t_n$, and abbreviate $\Psi(t_{n} + \frac{1}{2} \tau)$ as $\Psi^{n + \frac{1}{2}}$. We further introduce the extrapolation formula for $\Psi(t_{n} + \frac{1}{2} \tau)$, given by $\tilde{\Psi}^{n + \frac{1}{2}} = \frac{\tau_{n+1} + 2\tau_{n}}{2\tau_{n}}\Psi^{n} - \frac{\tau_{n+1}}{2\tau_{n}} \Psi^{n-1}$.

\subsection{A second-order ETD mr-SAV multistep (ETD-mr-SAV-MS2-L) scheme}
In order to obtain a second-order accuracy approximation $\omega^{n+1}$ and $q^{n+1}$, we approximate the nonlinear terms and external force terms in Eq. \eqref{eqn:mr_sav} with a second-order midpoint extrapolations. 
For $t \in (t_n, t_{n+1}]$, the resulting approximation system is given by
\begin{subequations}\label{eqn:2nd_mrSAV}
    \begin{align}
    &\frac{\partial \omega}{\partial t} + \nu \mathcal{L} \omega + q(t) B(\tilde{\omega}^{n+\frac{1}{2}}, \tilde{\omega}^{n +\frac{1}{2}}) = f^{n+\frac{1}{2}}, \label{eqn:2nd_mrSAV_u}\\
    &\frac{d q}{dt} + \gamma q  = \langle B(\tilde{\omega}^{n+\frac{1}{2}}, \tilde{\omega}^{n+\frac{1}{2}}), \omega \rangle + \gamma. \label{eqn:2nd_mrSAV_q}
    \end{align}
\end{subequations}
The nonlinear term could also be approximated by $\tilde{B}^{n+\frac12}:=  \frac{\tau_{n+1} + 2\tau_{n}}{2\tau_{n}}B(\omega^{n},\omega^n) - \frac{\tau_{n+1}}{2\tau_{n}} B(\omega^{n-1},\omega^{n-1})$.

By the variation of constants formula, equation \eqref{eqn:2nd_mrSAV_u} implies that
\begin{equation}\label{eqn:2nd_etd_mrSAV_ms_u} 
    \omega(t_n+\tau) = \varphi_0(\tau \nu \mathcal{L})\omega^{n} + \tau \varphi_1(\tau\nu \mathcal{L})f^{n+\frac{1}{2}} -  \int_{0}^{\tau} e^{-(\tau-s)\nu \mathcal{L}} q(s+t_n) B(\tilde{\omega}^{n+\frac{1}{2}},\tilde{\omega}^{n+\frac{1}{2}}) ds,
\end{equation}
where $\tau := t - t_n$, and the functions $\varphi_0(\cdot)$ and $\varphi_1(\cdot)$ are defined as
\begin{equation}\label{eqn:exp_relat}
    \varphi_0(z) = \mathrm{e}^{-z},\quad \varphi_1(z) = \frac{1 - \mathrm{e}^{-z}}{z}.
\end{equation}

We recover the following classical {\bf ETD-MS2} scheme \cite{zhang1987schema}
\begin{equation}\label{eqn:ETD-MS2} 
    \omega(t_n+\tau) = \varphi_0(\tau \nu \mathcal{L})\omega^{n} + \tau \varphi_1(\tau\nu \mathcal{L})f^{n+\frac{1}{2}} -  \tau\varphi_1(\tau\nu\mathcal{L}) B(\tilde{\omega}^{n+\frac{1}{2}},\tilde{\omega}^{n+\frac{1}{2}}),
\end{equation}
if we impose $q(t)\equiv 1$ in \eqref{eqn:2nd_etd_mrSAV_ms_u}. Next, we substitute this expression for $\omega(t)$ into \eqref{eqn:2nd_mrSAV_q} and deduce
\begin{equation}\label{eqn:2nd_etd_mrSAV_ms_q} 
\begin{aligned}
    &\frac{d q(\tau+t_n)}{d\tau} + \gamma q(\tau + t_n) +  \langle B(\tilde{\omega}^{n+\frac{1}{2}},\tilde{\omega}^{n+\frac{1}{2}}), \int_{0}^{\tau} e^{-(\tau-s)\nu \mathcal{L}} q(s+t_n) B(\tilde{\omega}^{n+\frac{1}{2}},\tilde{\omega}^{n+\frac{1}{2}}) ds \rangle   \\
    =& \langle B(\tilde{\omega}^{n+\frac{1}{2}}, \tilde{\omega}^{n+\frac{1}{2}}), \varphi_0(\tau \nu \mathcal{L})\omega^{n} + \tau \varphi_1(\tau\nu \mathcal{L})f^{n+\frac{1}{2}} \rangle + \gamma.
\end{aligned}
\end{equation}
This linear differential-integral equation of convolution type with $\mathcal{L}$ generating a $C^0$-semigroup on $L^2$ is uniquely solvable \cite{Brunner2017}, provided $B(\tilde{\omega}^{n+\frac{1}{2}},\tilde{\omega}^{n+\frac{1}{2}})\in L^2$.
 {%\color{red}
 Moreover, $q$ is analytic in $t$ if $f^{n+\frac12}$ is Gevrey class regular}  \cite{Brunner2017, constantin1988navier}.
%\footnote{This is guaranteed as long as $\omega^n, \omega^{n-1}\in H^1$}.
\ignore{\begin{equation}\label{eqn:2nd_etd_mrSAV_ms_q} 
\begin{aligned}
    &\frac{d q(t)}{dt} + \gamma q(t) +  \langle B(\tilde{\omega}^{n+\frac{1}{2}},\tilde{\omega}^{n+\frac{1}{2}}), \int_{0}^{\tau}  e^{-(\tau-s)\nu \mathcal{L}} q(s+t_n) B(\tilde{\omega}^{n+\frac{1}{2}},\tilde{\omega}^{n+\frac{1}{2}}) ds \rangle   \\
    =& \langle B(\tilde{\omega}^{n+\frac{1}{2}}, \tilde{\omega}^{n+\frac{1}{2}}), \varphi_0(\tau \nu \mathcal{L})\omega^{n} + \tau \varphi_1(\tau\nu \mathcal{L})f^{n+\frac{1}{2}} \rangle + \gamma.
\end{aligned}
\end{equation}}
Since the integral in \eqref{eqn:2nd_etd_mrSAV_ms_q} is in the form of a convolution, it can be solved via the Laplace transformation. We thus derive the following expression for the Laplace transform of $q$, denoted $\hat{q}(z)$:
\begin{equation}\label{eqn:q-Laplace}
\begin{aligned}
     &(z  + \gamma + \langle B(\tilde{\omega}^{n+\frac{1}{2}},\tilde{\omega}^{n+\frac{1}{2}}), (z + \nu\mathcal{L})^{-1} B(\tilde{\omega}^{n+\frac{1}{2}},\tilde{\omega}^{n+\frac{1}{2}}) \rangle) \hat{q}(z)   \\
    =& q(t_n) + \frac{1}{z} \gamma + \langle (z + \nu\mathcal{L})^{-1}B(\tilde{\omega}^{n+\frac{1}{2}}, \tilde{\omega}^{n+\frac{1}{2}}),  \omega^{n} + \frac{f^{n+\frac{1}{2}}}{z} \rangle.
\end{aligned}
\end{equation}
\ignore{\begin{equation}\label{eqn:q-Laplace}
\begin{aligned}
     &(z  + \gamma + \langle B(\tilde{\omega}^{n+\frac{1}{2}},\tilde{\omega}^{n+\frac{1}{2}}), (z - \tau \mathcal{L})^{-1} B(\tilde{\omega}^{n+\frac{1}{2}},\tilde{\omega}^{n+\frac{1}{2}}) \rangle) \hat{q}(z)   \\
    =& q(t_n) +  \frac{1}{z} \left(\langle B(\tilde{\omega}^{n+\frac{1}{2}}, \tilde{\omega}^{n+\frac{1}{2}}), \varphi_0(\tau \nu \mathcal{L})\omega^{n} +  \tau \varphi_1(\tau\nu \mathcal{L})f^{n+\frac{1}{2}} \rangle + \gamma\right).
\end{aligned}
\end{equation}}
We can express $\hat{q}$, the Laplace transform of $q$, concisely as:
%The Laplace transform equation can then be written concisely as:
\begin{equation}\label{eqn:q-hat-explicit}
\hat{q}(z) = \frac{1}{\Gamma(z)} \left[ q(t_n) + \frac{\gamma}{z} + \mathcal{G}(z) \right],
\end{equation}
\noindent where the functions are explicitly defined as:
\begin{equation}\label{eqn:q_hat_formula}
\begin{aligned}
&\Gamma(z) = z + \gamma + \langle B(\tilde{\omega}^{n+\frac{1}{2}},\tilde{\omega}^{n+\frac{1}{2}}), (z + \nu\mathcal{L})^{-1} B(\tilde{\omega}^{n+\frac{1}{2}},\tilde{\omega}^{n+\frac{1}{2}}) \rangle, \\
&\mathcal{G}(z) = \langle (z + \nu\mathcal{L})^{-1} B(\tilde{\omega}^{n+\frac{1}{2}}, \tilde{\omega}^{n+\frac{1}{2}}),  \omega^{n} +\frac{f^{n+\frac{1}{2}}}{z} \rangle.
\end{aligned}
\end{equation}

Notice that when $B(\tilde{\omega}^{n+\frac{1}{2}},\tilde{\omega}^{n+\frac{1}{2}})\in L^2$, the singularities of $\hat{q}$ are contained in a band around the negative real-axis. Hence, we can resort to Talbot's algorithm to recover $q$ for {%\color{red}
a few} $t$ values within the interval $(t_n, t_{n+1}]$ so that the integral in \eqref{eqn:2nd_etd_mrSAV_ms_u} can be approximated %{\cancel{up to order $\tau^3_{n+1}$}}
%\color{red}
{{up to machine precision with a quadrature rule using the values of $q$ at these points assuming that the time-step is not big, say $\tau\le 10^{-2}$.}}
%{\cancel{ for $t=t_{n+1}$ since $q$ is analytic in time.}} 
The Talbot's algorithm converges exponentially, and hence only a handful points are needed for the error in computing the $q$ values to reach machine precision for a well-chosen contour and points \cite{trefethen2006talbot}.  Ten points are untilized in our computation. Subsequently, the solution $\omega^{n+1}$ at the next time step can be computed from equation \eqref{eqn:2nd_etd_mrSAV_ms_u} {%\color{red} 
to machine precision essentially.}
%{\cancel{with local truncation error of the order of $\tau^3_{n+1}$.}} 
{%\color{red} 
We also point out that the computation of $\hat{q}(z)$ for different $z$ values uses the same nonlinear term $B(\tilde{\omega}^{n+\frac{1}{2}},\tilde{\omega}^{n+\frac{1}{2}})$. Therefore, only two FFT are required to compute the $\hat{q}$ values. Furthermore, the computation of $q(t)$ for different values of $t$ uses the same $\hat{q}(z)$. Henceforth, the overhead of computing $q(t)$ for different $t$ beyond the first one is very small.}

The scheme proposed herein is referred to as the second-order ETD mr-SAV multistep \textbf{(ETD-mr-SAV-MS2-L)} scheme, where the suffix L emphasizes the role of Laplace transform, and to distinguish from a similar second-order ETD-mr-SAV-MS2 scheme proposed in a recent work \cite{WangWangZhang2026} where the scalar auxiliary variable solves a nonlinear algebraic equation. 

We next comment on the computational complexity of the three schemes. The ETD-MS2, ETD-mr-SAV-MS2-L, and ETD-mr-SAV-MS2 schemes all require the evaluation of the same explicitly treated nonlinear term. In a Fourier pseudo-spectral implementation, this step has computational complexity of order $\mathcal{O}(N\log N)$, where $N$ denotes the total number of spatial degrees of freedom.

The main additional cost of the ETD-mr-SAV-MS2-L scheme comes from the use of the Laplace-transform representation and its numerical inversion. In particular, the Talbot-type quadrature used for the inverse Laplace transform requires a number of diagonal matrix-vector multiplications, as well as the corresponding vector inner products. Each such diagonal matrix-vector multiplication or inner product has cost $\mathcal{O}(N)$, which is asymptotically lower than the $\mathcal{O}(N\log N)$ cost of evaluating the nonlinear term. However, the numerical inversion typically involves on the order of two dozen such operations. Consequently, although these are lower-order operations from the asymptotic viewpoint, their cumulative contribution is not negligible for moderate values of $N$, as confirmed by the numerical results reported below.

Thus, compared with ETD-MS2 and ETD-mr-SAV-MS2, the ETD-mr-SAV-MS2-L scheme has a larger computational prefactor. This additional cost is the price paid for its stronger structural property: among the three schemes considered here, only ETD-mr-SAV-MS2-L simultaneously guarantees linearity, unique solvability for all parameter values, and uniform-in-time stability.

\ignore{
\begin{rem}
    The inverse Laplace transform at each time step, needed for the recovery of $q(t)$, is relatively computationally time-consuming. A more tractable second-order ETD mr-SAV multistep scheme is to approximate $q(s+t_n)$ in the integral term of \eqref{eqn:2nd_etd_mrSAV_ms_q} with an $r$-th order Lagrange interpolation using known information from the previous steps. This modification makes the scheme much simpler and shows better performance in our numerical experiments. We refer to this scheme as \textbf{(ETD-mr-SAV-ms2-r)}. Nevertheless, due to the introduction of interpolation, the rigorous verification of its stability is still unknown. 
\end{rem}
}

%\newpage
\subsection{Adaptive time-step ETD mr-SAV multistep scheme}
We now present an embedded adaptive time-stepping algorithm based on the long-time stable variable-step second-order scheme \eqref{eqn:2nd_mrSAV}. This embedded adaptive scheme enables automatic error control without significant computational overhead while maintaining long-time stability. See \cite{Fehlberg1969NASA, DormandPrince1980JCAM} for some historical notes on the development of embedded pairs, and 
\cite{decaria2021embedded} for a recent interesting 1--2 pair for the 2D Navier-Stokes equations (without the $L^\infty(0,\infty;L^2)$ bound).

First, in the implementation of the ETD-mr-SAV-ms2 scheme, we employ the $k$-points Gauss Lobatto quadrature rule to approximate the integral term arising in the solution of Equation \eqref{eqn:2nd_etd_mrSAV_ms_u}. 
Accordingly, equation \eqref{eqn:2nd_etd_mrSAV_ms_u} is solved using the formulation
\begin{equation}\label{eqn:etd_mrSAV_quad}
\begin{aligned}
\omega^{n+1} = & \varphi_0(\tau_{n+1} \nu \mathcal{L})\omega^{n} + \tau_{n+1} \varphi_1(\tau_{n+1}\nu \mathcal{L})f^{n+\frac{1}{2}} \\
&-\tau_{n+1} \sum_{i=1}^{k} m_i e^{-(\tau_{n+1}-s_i)\nu \mathcal{L}} q(s_i+t_n) B(\tilde{\omega}^{n+\frac{1}{2}},\tilde{\omega}^{n+\frac{1}{2}}),
\end{aligned}
\end{equation}
where $s_i$ is the quadrature points and $m_i$ is the corresponding weights.
To ensure that the solution formula achieves machine precision while maintaining computational efficiency, the \(6\)-point Gauss--Lobatto quadrature rule is employed in the subsequent numerical experiments. 

Following the classical SAV technique, the above second-order exponential time-differencing mean-reverting SAV two-step (ETD-mr-SAV-MS2-L) scheme can be solved efficiently. The procedure is as follows:
\begin{enumerate}
    \item Calculate the auxiliary variable $q(t_n+s)$ via Talbot's algorithm:
    \begin{equation}\label{eqn:step_1}
        q_i = q(t_n+s_i).
    \end{equation}
    \item Calculate the intermediate variables:
    \begin{equation}\label{eqn:step_2}
        \begin{aligned}
        & v^{n+1}_1 = \varphi_0(\tau_{n+1} \nu \mathcal{L})\omega^{n}
        + \tau_{n+1} \varphi_1(\tau_{n+1} \nu \mathcal{L})f^{n+\frac{1}{2}},\\
        & v^{n+1}_{2,i} = \tau_{n+1}e^{-(\tau_{n+1}-s_i)\nu \mathcal{L}} q_i
        B(\tilde{\omega}^{n+\frac{1}{2}}, \tilde{\omega}^{n+\frac{1}{2}}),\quad i = 1,\dots,k.
        \end{aligned}
    \end{equation}
    % \item Compute the coefficients:
    % \begin{equation}\label{eqn:step_2}
    %     A^{n+1} = \langle \omega_1^{n+1}, \omega_2^{n+1}\rangle,\quad
    %     B^{n+1} = \Vert \omega_2^{n+1} \Vert^2,\quad
    %     C^{n+1} = \varphi_0(\tau_{n+1} \gamma)\, r^n.
    % \end{equation}
    % \item Determine the auxiliary variable $r^{n+1}$ by finding the smallest root of the cubic polynomial
    % \begin{equation}\label{eqn:cubic}
    %    g(r)=B^{n+1} r^3 - B^{n+1} r^2 + (1 + A^{n+1} - B^{n+1}) r -(A^{n+1} - B^{n+1} + C^{n+1}).
    % \end{equation}
    \item Obtain $\omega^{n+1}$ from
    \begin{equation}\label{eqn:step_3}
         \omega^{n+1} = v_1^{n+1} - \sum_{i=1}^{k} m_i v_{2,i}^{n+1}.
    \end{equation}
\end{enumerate}
In all these calculations, we pre-compute $B(\tilde{\omega}^{n+\frac{1}{2}},\tilde{\omega}^{n+\frac{1}{2}})$ so that the nonlinear term (and the associated FFT) are evaluated only once.

Next, when employing an adaptive algorithm, the local error must be evaluated using a first-order accurate numerical scheme, for which a first-order one-level ETD-mr-SAV scheme can be utilized. To avoid incurring additional computational overhead, we adopt the embedded strategy to construct the first-order scheme.

To formulate a first-order embedded numerical scheme, we apply the left rectangular quadrature rule to solve equation \eqref{eqn:2nd_etd_mrSAV_ms_u}, yielding the following {\bf embedded first-order scheme}
\begin{equation}\label{eqn:step_4}
\begin{aligned}
   \omega^{n+1}_{(1)} = & \varphi_0(\tau_{n+1} \nu \mathcal{L})\omega^{n} + \tau_{n+1} \varphi_1(\tau_{n+1}\nu \mathcal{L})f^{n+\frac{1}{2}} -\tau_{n+1} q(t_{n+1}) B(\tilde{\omega}^{n+\frac{1}{2}},\tilde{\omega}^{n+\frac{1}{2}}) \\
\end{aligned}
\end{equation}
It is straightforward to verify that ${\omega}^{n + 1}_{(1)}$ attains only first-order accuracy. This scheme is solved in the same manner as \eqref{eqn:2nd_etd_mrSAV_ms_u} and uses the same intermediate variables that arise in \eqref{eqn:etd_mrSAV_quad}, incurring no additional computational cost.

% \begin{subequations}\label{eqn:mrSAV_ETD1}
%     \begin{align}
%     &\bar{\omega}^{n+1} = \varphi_0(\tau_{n+1} \nu \mathcal{L})\omega^{n} + \tau_{n+1} \varphi_1(\tau_{n+1} \nu \mathcal{L})(f^{n} - q^{n+1} B(\omega^{n},\omega^{n})), \label{eqn:mrSAV_ETD1_u}\\
%     &q^{n+1} = \varphi_0(\tau_{n+1} \gamma) q^{n} + \tau_{n+1} \varphi_1(\tau_{n+1} \gamma) ( \langle B(\omega^{n}, \omega^{n}), \bar{\omega}^{n+1}\rangle + \gamma), \label{eqn:mrSAV_ETD1_q}
%     \end{align}
% \end{subequations}

% We therefore refer to \eqref{eqn:mrSAV_ETDMS1} as the first-order exponential time-differencing mean-reverting SAV (\textbf{ETD-mr-SAV-MS2-1o}) scheme.

% Owing to its structural similarity with the ETD-mr-SAV-MS2 scheme, the intermediate values employed in solving \eqref{eqn:2nd_etd_mrSAV_ms_u} can be reused for the computation of \textbf{ETD-mr-SAV-MS2-1o}. In particular, the variable \(\bar{\omega}^{n+1}\) can be obtained via the expression
% \begin{equation}\label{eqn:step_4}
%     \bar{\omega}^{n+1} = \bm v_1^{n+1} - \bm v_{2,3}^{n+1}.
% \end{equation}

Combining the first- and second-order ETD-mr-SAV schemes above, we construct an embedded adaptive algorithm. Following \cite{hairer1993solving}, the adaptive strategy is summarized in {\bf Algorithm~\ref{alg:adaptive}} with the time-step update factor function
\begin{equation}\label{update_function}
A_{dp}(e_{\omega}, e_q, \rho, s)
= \rho \left(\min\left\{\frac{\text{tol}}{e_{\omega}},\ \frac{\text{tol}}{e_q}\right\}\right)^{\frac{1}{s}}.
\end{equation}
Here, $\rho$ denotes a safety factor, and $\text{tol}$ represents the prescribed/desired tolerances for the vorticity variable $\omega$ and the scalar auxiliary variable $q$, respectively. These parameters depend on the specific problem under consideration. We refer to this adaptive time-stepping method as the \textbf{ETD-mr-SAV-MS12-L} scheme.

% Here, $\rho$ is a safety factor, and $\text{tol}_{\omega}$ and $\text{tol}_q$ are prescribed tolerances for the vorticity variable $\omega$ and the scalar auxiliary variable $q$, respectively. These parameters depend on the specific problem under consideration. We refer to this adaptive time-stepping method as the \textbf{ETD-mr-SAV-MS12} scheme.

\begin{algorithm}[ht]
\footnotesize
\caption{ETD-mr-SAV-MS12-L scheme}
\label{alg:adaptive}
\begin{algorithmic}
\State \textbf{Given:} $\omega^{n}, \omega^{n-1}$, $q^{n}$, $\tau_{n+1}$.
\State \textbf{Step 1.} Compute $q_i$, $i = 1,\dots,k$, via the Talbot's algorithm.
\State \textbf{Step 2.} Compute $\omega_1^{n+1}$, $\omega_{2,i}^{n+1}$, $i=1, \dots, k$, via \eqref{eqn:step_2}. 
\State \textbf{Step 3.} Compute the second-order approximation $\omega_{(2)}^{n+1}$ and the first-order approximation $\omega_{(1)}^{n+1}$ via \eqref{eqn:step_3} and \eqref{eqn:step_4}, respectively.
\State \textbf{Step 4.} Compute
$e_{\omega}^{n+1} = \frac{\Vert \omega_{(2)}^{n+1} - \omega_{(1)}^{n+1} \Vert}{\ \Vert \omega_{(2)}^{n+1} \Vert}$
and $e_q^{n+1}=|q(t_{n+1})-1|$.
\If{$\max\{e_{\omega}^{n+1},e_{q}^{n+1} \} \le \text{tol}$}
    \State \textbf{Step 5.} Set $\tau_{n+2} \leftarrow max\{\tau_{\min},\ \min\{A_{dp}(e_{\omega}^{n+1}, e_q^{n+1},\rho,2)\tau_{n+1},\ \tau_{\max}\}\}$, $\omega^{n+1} = \omega_{(2)}^{n+1}$ and  $q^{n+1} = q(t_{n+1})$.
\Else
    \State \textbf{Step 6.} Reject $\omega_{(2)}^{n+1}, \tau_{n+1}$. Reset
    $\tau_{n+1} \leftarrow \max\{\tau_{\min},\ \min\{A_{dp}(e_{\omega}^{n+1}, e_q^{n+1}, \rho, 2) \tau_{n+1},\ \tau_{\max}\}\}$.
    \State \textbf{Step 7.} Go to \textbf{Step 1}.
\EndIf
\end{algorithmic}
\end{algorithm}

\ignore{
\begin{rem}
In this adaptive time-stepping setting, one may instead use the first-order ETD-mr-SAV scheme \eqref{eqn:mrSAV_ETD1} to obtain $\bar{\omega}^{n+1}$. However, to construct an embedded adaptive algorithm without additional computational cost, we choose the first-order two-step ETD-mr-SAV scheme \eqref{eqn:step_4} so that the intermediate computations can be shared with the ETD-mr-SAV-MS2 scheme.
\end{rem}
}

%\newpage
\subsection{Adaptive VSVO ETD mr-SAV scheme}

In this section, we propose an adaptive variable-stepsize, variable-order (VSVO) time-stepping scheme based on the variable-step methods presented earlier. This scheme enables the automatic adjustment of both the temporal order and step size depending on the local truncation error. %, while maintaining long-term stability.

Following the work of \cite{li2020adaptive,decaria2021embedded}, we combine the embedded first- and second-order ETD-mr-SAV schemes detailed above together with a variable-step BDF2 scheme to come up with our VSVO strategy outlined in {\bf Algorithm~\ref{alg:VSVO}}.

Here, we adopt the same parameter selection as \cite{decaria2021embedded}. The aforementioned algorithm contains estimates of the local error within Step 4. Specifically, $EST_1$ provides an estimate of the local error associated with the first-order approximation $\omega^{n+1}_{(1)}$. For the second-order error estimator $EST_2$, it is derived by substituting the solution obtained from the ETD-mr-SAV-MS2-L scheme into the residual of the variable-stepsize BDF2 formulation, as given by
\begin{equation*}
\begin{aligned}
\text{RES}_{\text{BDF2}}(\omega_{(2)}^{n+1})=&\frac{(1 + 2 r_{n+1})\omega_{(2)}^{n+1}-(1 + r_{n+1})^2\omega^{n}+r_{n+1}^2\omega^{n-1}}{(1 + r_{n+1})\tau_{n}}\\
&+\nu\mathcal{L}\omega_{(2)}^{n+1}+B(\omega_{(2)}^{n+1},\omega_{(2)}^{n+1})-f^{n+1},
\end{aligned}
\end{equation*}
where $r_{n+1}=\tau_{n+1}/\tau_{n}$ represents the reciprocal stepsize ratio. The differentiation and nonlinear term can be calculated efficiently in the Fourier space with the aid of FFT. Alternatively, we could also estimate $EST_2$ via a third-order interpolation formula \cite{decaria2021embedded}, expressed as
\begin{equation*}
\begin{aligned}
\text{RES}_{\text{Interp}}(\omega_{(2)}^{n+1}) =& \omega_{(2)}^{n+1}-\frac{(1 + r_{n+1})\bigl(1 + r_{n}(1 + r_{n+1})\bigr)}{1 + r_{n}}\omega^{n}+r_{n+1}\bigl(1 + r_{n}(1 + r_{n+1})\bigr)\omega^{n-1}\\
&-\frac{r_n^2r_{n+1}(1 + r_{n+1})}{1 + r_n}\omega^{n-2}.
\end{aligned}
\end{equation*}
This modification is also highly efficient, but converts the scheme into a three-step method, thereby increasing its memory complexity. We will use the BDF2 approach in this work.

\begin{algorithm}[htbp]
\footnotesize
\caption{VSVO scheme}
\label{alg:VSVO}
\begin{algorithmic}
\State \textbf{Given:} $\omega^{n}, \omega^{n-1}$, $q^{n}$, $\tau_{n+1}$.
\State \textbf{Step 1.} Compute $q_i$, $i=1,\cdots,k$ via the Talbot's algorithm.
\State \textbf{Step 2.} Compute $\omega_1^{n+1}$, $\omega_{2,i}^{n+1}$ ($i=1,\cdots, k$) via \eqref{eqn:step_2}. 
\State \textbf{Step 3.} Compute the second-order approximation $\omega_{(2)}^{n+1}$ and the first-order approximation $\omega_{(1)}^{n+1}$ via \eqref{eqn:step_3} and \eqref{eqn:step_4}, respectively.
\State \textbf{Step 4.} Compute the local-in-time errors in $\omega_{(1)}^{n+1}$ and $\omega_{(2)}^{n+1}$ using the following definitions:
\begin{equation*}
    \text{EST}_1 = \omega_{(2)}^{n+1} - \omega_{(1)}^{n+1},\quad \text{EST}_2 = \text{RES}_\text{BDF2}(\omega_{(2)}^{n+1})
\end{equation*}
    
\State \textbf{Step 5.} Compute \(e_{\omega,1}^{n+1} = \frac{\Vert EST_1 \Vert}{\Vert \omega_{(2)}^{n+1} \Vert}\), \(e_{\omega,2}^{n+1} = \frac{\Vert EST_2 \Vert}{\Vert \omega_{(2)}^{n+1} \Vert}\), and \(e_q^{n+1} = |q_k -1|\).
\If{$e_q^{n+1} \leq \text{tol}$ and $\min\{\Vert e_{\omega,1}^{n+1} \Vert, \Vert e_{\omega,2}^{n+1} \Vert \} \le \text{tol}$}
    \State \textbf{Step 6.} Compute 
    \begin{equation*}
        \tau_{n+2}^{(1)}= A_{dp}(e_{\omega,1}^{n+1}, e_q^{n+1}, 0.9,2)\tau_{n+1} ,\quad \tau_{n+2}^{(2)} = A_{dp}(e_{\omega,2}^{n+1}, e_q^{n+1}, 0.9,3)\tau_{n+1}
    \end{equation*}
    set
    \begin{equation*}
        i = \underset{i=\{1,2\}}{\arg\max}\ \tau_{n+2}^{(i)},\quad \tau_{n+2} = \max\{\tau_{\min}, \min\{ \tau_{n+2}^{(i)}, \tau_{\max} \}\},\quad \omega^{n+1} = \omega_{(i)}^{n+1},\quad q^{n+1} = q(t_{n+1})
    \end{equation*}
\Else
    \State \textbf{Step 7.} Set
    \begin{equation*}
        \tau_{n+1}^{(1)}= A_{dp}(e_{\omega,1}^{n+1}, e_q^{n+1}, 0.7,2)\tau_{n+1} ,\quad \tau_{n+1}^{(2)} = A_{dp}(e_{\omega,2}^{n+1}, e_q^{n+1}, 0.7,3)\tau_{n+1}
    \end{equation*}
    Set 
    \begin{equation*}
        i = \underset{i=\{1,2\}}{\arg\max}\ \tau_{n+1}^{(i)},\quad \tau_{n+1} = \max\{\tau_{\min}, \min\{ \tau_{n+1}^{(i)}, \tau_{\max} \}\}
    \end{equation*}
    \State {\textbf{Step 8.}} Go to \textbf{Step 1}.
\EndIf

\end{algorithmic}
\end{algorithm}

\section{Long-time stability of the ETD-mr-SAV schemes}\label{sec:4}

The purpose of this section is to derive a uniform-in-time enstrophy bound for the ETD-mr-SAV schemes proposed in the previous section. More precisely, we establish $L^\infty(0,\infty;L^2)$ stability, together with a corresponding cumulative dissipation estimate, under uniformly bounded forcing.

\medskip
We first recall several elementary properties of the exponential-related functions introduced in \eqref{eqn:exp_relat}.

\begin{lemma}\label{lem:phi_prop}
The functions $\varphi_i$ defined in \eqref{eqn:exp_relat} satisfy:
\begin{itemize}
    \item[(1)] $\varphi_i(z)$ is decreasing for $i=0,1$.
    \item[(2)] $z\varphi_1(z)+\varphi_0(z)=1$.
    \item[(3)] $0\le \varphi_0(z)\le 1$, $0\le \varphi_1(z)\le 1$, and $1+\frac{z}{2}\le (\varphi_1(z))^{-1}\le 1+z,\qquad \forall z\ge 0$.
    \item[(4)] $(\varphi_1(z))^{-1}\ge z$, \qquad $\forall z\in \mathbb{R}^+$.
\end{itemize}
\end{lemma}

\medskip
Using the Fourier representation, let $\lambda_{\bm k}$ denote the eigenvalue of the Stokes operator $\mathcal{L}$ associated with the Fourier mode $\hat{h}_{\bm k} e^{\mathrm{i}\bm k\cdot \bm x}$. Then, for any $h\in V$ and $\tau\ge 0$,
\begin{equation}\label{eqn:phi_op_L}
\begin{aligned}
\varphi_1(\tau \nu \mathcal{L})h
&=(\tau \nu \mathcal{L})^{-1}\bigl(I-e^{-\tau \nu \mathcal{L}}\bigr)h
=\sum_{\bm k\in \mathbb{Z}^2\setminus\{0\}}
\frac{1-e^{-\tau \nu \lambda_{\bm k}}}{\tau \nu \lambda_{\bm k}}\,
\hat{h}_{\bm k}e^{\mathrm{i}\bm k\cdot \bm x},\\
(\varphi_1(\tau \nu \mathcal{L}))^{-1}h
&=\bigl(I-e^{-\tau \nu \mathcal{L}}\bigr)^{-1}(\tau \nu \mathcal{L})h
=\sum_{\bm k\in \mathbb{Z}^2\setminus\{0\}}
\frac{\tau \nu \lambda_{\bm k}}{1-e^{-\tau \nu \lambda_{\bm k}}}\,
\hat{h}_{\bm k}e^{\mathrm{i}\bm k\cdot \bm x}.
\end{aligned}
\end{equation}
It follows that $(\varphi_1(\tau \nu \mathcal{L}))^{-1}$ is positive definite and commutes with $\mathcal{L}$. Moreover, for any $h, g\in V$,
\begin{equation}\label{eqn:sbp_formulas}
\langle \mathcal{L} h, g\rangle
= \langle \mathcal{L}^{1/2} h,\mathcal{L}^{1/2} g\rangle,\qquad
\langle \varphi_1(\tau \nu \mathcal{L})h, g\rangle
= \langle (\varphi_1(\tau \nu \mathcal{L}))^{1/2} h,(\varphi_1(\tau \nu \mathcal{L}))^{1/2} g\rangle,
\end{equation}
where $\mathcal{L}^{1/2}$ and $(\varphi_1(\tau \nu \mathcal{L}))^{1/2}$ are defined spectrally. In particular,
\begin{equation}\label{eqn:inv_expansion}
\begin{aligned}
\mathcal{L}^{1/2}h
&=\sum_{\bm k\in \mathbb{Z}^2\setminus\{0\}}
\lambda_{\bm k}^{1/2}\hat{h}_{\bm k}e^{\mathrm{i}\bm k\cdot \bm x},\quad 
(\varphi_1(\tau \nu \mathcal{L}))^{1/2}h
=\sum_{\bm k\in \mathbb{Z}^2\setminus\{0\}}
\left(\frac{1-e^{-\tau \nu \lambda_{\bm k}}}{\tau \nu \lambda_{\bm k}}\right)^{1/2}
\hat{h}_{\bm k}e^{\mathrm{i}\bm k\cdot \bm x}.
\end{aligned}
\end{equation}
We also recall the Poincar\'e inequality
\begin{equation}\label{eqn:poincare}
\Vert h \Vert \leq \frac{1}{\sqrt{\lambda_1}} \Vert \mathcal{L}^{1/2}h \Vert,\qquad \forall h\in V,
\end{equation}
where $\lambda_1>0$ is the smallest eigenvalue of $\mathcal{L}$.

\medskip
Combining the spectral representation \eqref{eqn:inv_expansion} with Lemma~\ref{lem:phi_prop}, we obtain the following operator inequalities.

\begin{lemma}\label{lem:inv_opest}
For any $h\in V$ and $\tau\ge 0$, the following estimates hold:
\begin{equation}\label{eqn:op_est_1}
\Vert h \Vert^2
= \tau \nu \Vert (\varphi_1(\tau \nu \mathcal{L}))^{1/2}\mathcal{L}^{1/2} h \Vert^2
+ \Vert (\varphi_0(\tau \nu \mathcal{L}))^{1/2} h \Vert^2,
\end{equation}
\begin{equation}\label{eqn:op_est_2}
\Vert (\varphi_1(\tau \nu \mathcal{L}))^{1/2}h \Vert^2 \le \Vert h \Vert^2,
\end{equation}
\begin{equation}\label{eqn:inv_op_est_1}
\Vert h \Vert^2 + \frac{\tau \nu}{2}\Vert \mathcal{L}^{1/2} h \Vert^2
\le \Vert (\varphi_1(\tau \nu \mathcal{L}))^{-1/2} h \Vert^2
\le \Vert h \Vert^2 + \tau \nu \Vert \mathcal{L}^{1/2} h \Vert^2,
\end{equation}
and
\begin{equation}\label{eqn:inv_op_est_2}
\Vert (\varphi_1(\tau \nu \mathcal{L}))^{-1/2} h \Vert^2
\ge \tau \nu \Vert \mathcal{L}^{1/2} h \Vert^2.
\end{equation}
\end{lemma}
\subsection{Long-time stability for the ETD-mr-SAV-MS2 scheme}

We now derive an $L^\infty(0,\infty;L^2)$ bound for the ETD-mr-SAV-MS2 scheme \eqref{eqn:2nd_mrSAV}.

\begin{theorem}\label{thm:etd_sav_L_stability}
Assume the initial data $(\omega^i,q^i)\in  V\times\mathbb{R} $ for $i=0,1$, and the forcing term $f\in L^\infty(0,\infty; H)$. 
% Then there exists a constant $\theta>0$, independent of the time steps, such that for all $n\ge 1$, 
The ETD-mr-SAV-MS2-L scheme \eqref{eqn:2nd_mrSAV} is uniquely solvable with $\omega^n\in V\ \forall n\ge 1$, and $B(\tilde{\omega}^{n+\frac{1}{2}},\tilde{\omega}^{n+\frac{1}{2}})\in L^2$. Moreover, it is long-time stable in the sense that
\begin{equation}\label{eqn:L2_estimate_2}
\Vert \omega^{n+1} \Vert^2 + |q^{n+1}|^2
\le e^{-\theta \sum_{i=1}^{n}\tau_{i+1}}\bigl(\Vert \omega^1\Vert^2+|q^1|^2\bigr)
+\frac{1}{\theta}\left(\frac{1}{\lambda_1\nu}\Vert f\Vert_{L^\infty(0,\infty;H)}^2+\gamma\right),
\end{equation}
where
$$
\theta := \min\{\nu\lambda_1,\gamma\}>0.
$$
In addition, for any $1\le n\le N-1$, the following cumulative dissipation estimate holds:
\begin{equation}\label{eqn:L2H1_estimate}
\Vert \omega^{n+1} \Vert^2 + |q^{n+1}|^2
+ \sum_{k=1}^{n}\int_{t_k}^{t_{k+1}}\Bigl(\nu\Vert \mathcal{L}^{1/2}\omega(t)\Vert^2 + \gamma |q(t)|^2\Bigr)\,dt
\le \Vert \omega^{1} \Vert^2 + |q^{1}|^2
+ \frac{1}{\lambda_1 \nu} \Vert f \Vert_{L^2(0,T;H)}^2 +  \gamma T,
\end{equation}
where $T=\sum_{i=1}^{n+1}\tau_{i}$.
% Moreover, for any $1\le n\le N-1$, the following cumulative dissipation estimate holds:
% \begin{equation}\label{eqn:L2H1_estimate}
% \Vert \omega^{n+1} \Vert^2
% + \sum_{k=1}^{n}\int_{t_k}^{t_{k+1}}\Bigl(\nu\Vert \mathcal{L}^{1/2}\omega(t)\Vert^2 + \gamma |q(t)|^2\Bigr)\,dt
% \le \Vert \omega^{1} \Vert^2
% + T\left(\frac{1}{\lambda_1 \nu} \Vert f \Vert_{L^2(0,T;\bm H)}^2 +  \gamma\right).
% \end{equation}
\end{theorem}

\begin{proof}
We first work on a single time interval $t\in (t_n,t_{n+1}]$ and then iterate the resulting estimate.

Since  $\omega^n,\omega^{n-1}\in V $, we have
$B(\tilde{\omega}^{n+\frac{1}{2}},\tilde{\omega}^{n+\frac{1}{2}})\in L^{2}$.
Hence, the scalar auxiliary variable $q(t)$ determined by the differential-integral equation \eqref{eqn:2nd_etd_mrSAV_ms_q} of convolution type is uniquely solvable since $\mathcal{L}=-\Delta$ generates a $C_0$-semigroup on $\dot{L}^2=H$ \cite{Brunner2017}. Henceforth, the vorticity equation \eqref{eqn:2nd_mrSAV_u} is uniquely solvable with the solution at $t_{n+1}$
belonging to $V\cap\dot{H}^2_{per}$ for all $t\in (t_n,t_{n+1}]$. Consequently, we deduce that $\omega^n\in V\quad \forall n\ge 1$. This completes the unique solvability proof.

For the long-time stability, we take the inner product of \eqref{eqn:2nd_mrSAV_u} with $\omega(t)$ in $H$, and add it to the product of \eqref{eqn:2nd_mrSAV_q} with $q(t)$. Using the cancellation of the coupled terms (by design), we obtain
\begin{equation}\label{eqn:inner_2_mrSAV_u}
\frac12 \frac{d}{dt}\Vert \omega \Vert^2 + \nu\Vert \mathcal{L}^{1/2}\omega\Vert^2
+\frac12 \frac{d}{dt}|q(t)|^2 + \gamma|q(t)|^2
= \langle f^{n+\frac{1}{2}},\omega\rangle + \gamma q(t).
\end{equation}
By the Cauchy--Schwarz inequality and the Poincar\'e inequality \eqref{eqn:poincare},
\begin{equation}\label{eqn:2_mrSAV_est}
\frac{d}{dt}\bigl(\Vert \omega \Vert^2 + |q|^2\bigr)
+ \nu\Vert \mathcal{L}^{1/2}\omega\Vert^2 + \gamma |q|^2
\le \frac{1}{\lambda_1 \nu}\Vert f^{n+\frac{1}{2}}\Vert^2 + \gamma.
\end{equation}
Applying again \eqref{eqn:poincare} to bound $\nu\Vert \mathcal{L}^{1/2}\omega\Vert^2$ from below by $\nu\lambda_1\Vert \omega\Vert^2$, and setting
\begin{equation}\label{eqn:theta}
\theta=\min\{\nu\lambda_1,\gamma\}>0,
\end{equation}
we obtain
\begin{equation}\label{eqn:2_mrSAV_est_2}
\frac{d}{dt}\Bigl(e^{\theta t}\bigl(\Vert \omega \Vert^2 + |q|^2\bigr)\Bigr)
\le e^{\theta t}\left(\frac{1}{\lambda_1\nu}\Vert f^{n+\frac{1}{2}}\Vert^2 + \gamma\right).
\end{equation}
Integrating \eqref{eqn:2_mrSAV_est_2} over $(t_n,t_{n+1}]$ yields
\begin{equation}\label{eqn:2_mrSAV_est_3}
\Vert \omega^{n+1} \Vert^2 + |q^{n+1}|^2
\le e^{-\theta \tau_{n+1}}\bigl(\Vert \omega^{n}\Vert^2+|q^{n}|^2\bigr)
+\frac{1-e^{-\theta \tau_{n+1}}}{\theta}\left(\frac{1}{\lambda_1 \nu}\Vert f^{n+\frac{1}{2}}\Vert^2 + \gamma\right).
\end{equation}
Iterating the above inequality gives \eqref{eqn:L2_estimate_2}.

For \eqref{eqn:L2H1_estimate}, we integrate \eqref{eqn:2_mrSAV_est} over $(t_k,t_{k+1}]$ and sum from $k=1$ to $n$ to obtain
\begin{equation}
\begin{aligned}
&\Vert \omega^{n+1}\Vert^2 + |q^{n+1}|^2
+\sum_{k=1}^n\int_{t_k}^{t_{k+1}}\Bigl(\nu\Vert \mathcal{L}^{1/2}\omega(t)\Vert^2+\gamma|q(t)|^2\Bigr)\,dt \\
\le & \Vert  \omega^1\Vert^2 + |q^1|^2
+\sum_{k=1}^n\int_{t_k}^{t_{k+1}}\left(\frac{1}{\lambda_1\nu}\Vert f^{k+\frac12}\Vert^2+\gamma\right)\,dt,
\end{aligned}
\end{equation}
which implies \eqref{eqn:L2H1_estimate}.
\end{proof}

\begin{remark}

    We also observe that the solution can be more regular provided $f$ is more regular. Indeed, \eqref{eqn:2nd_mrSAV_u} and elliptic regularity imply $\omega^{n+1}\in \dot{H}^{m+1}_{per}$ provided $\omega^n, \omega^{n-1}\in \dot{H}^m_{per}, f^{n+\frac12}\in \dot{H}^{m-1}_{per}$. This further implies that $q\in C^m([t_n, t_{n+1}])$ and analytic on $(t_n, t_{n+1})$ by \eqref{eqn:2nd_etd_mrSAV_ms_q} \cite{Brunner2017}. In our computation, we will choose analytic initial data and external forcing. Hence the smoothness of our numerical solution is guaranteed.

\end{remark}

% Next, We derive an $L^\infty(0,\infty;H^{1})$ bound for the ETD-mr-SAV-MS2 scheme \eqref{eqn:2nd_mrSAV}.

% \begin{thm}
    
% \end{thm}

% \begin{proof}

% \begin{equation}
% \begin{aligned}
%     \omega^{n+1} = & \varphi_0(\tau_{n+1} \nu \mathcal{L})\omega^{n} + \tau \varphi_1(\tau\nu \mathcal{L})f^{n+\frac{1}{2}} -  \int_{0}^{\tau_{n+1}} e^{-(\tau_{n+1}-s)\nu \mathcal{L}} q(s+t_{n+1}) B(\tilde{\omega}^{n+\frac{1}{2}},\tilde{\omega}^{n+\frac{1}{2}}) ds,\\
% \end{aligned}
% \end{equation}

% \end{proof}

\section{Numerical experiments}\label{sec:5}

In this section, we report results of several numerical experiments that illustrate the temporal accuracy, long-time stability, and practical performance of the proposed second-order scheme, including its usefulness as a building block for two time-adaptive strategies, a variable-step approach, and a variable-step-variable-order approach. In particular, we verify the expected temporal convergence order under variable time stepping (in the spirit of \cite{chen2019variable}), demonstrate the ability of the adaptive strategy to adjust time-step sizes according to the speed of temporal evolution while maintaining a prescribed error tolerance, and test long-time stability in a Kolmogorov forcing setting. For the Kolmogorov flow, we further provide a comparative study of different schemes in terms of long-time boundedness and statistical consistency of key physical quantities.

All simulations are performed for the two-dimensional NSE on a $2\pi$ periodic box. A Fourier spectral method is used for spatial discretization. Since solutions to the NSE are analytic in space under spatially analytic forcing \cite{constantin1988navier}, the spatial discretization error is negligible compared with the temporal discretization error due to the spectral accuracy of the Fourier method. 

\subsection{Convergence test}
Here we test the rate of convergence of the ETD-mr-SAV-MS2-L scheme using both fixed and variable time-steps. We also compare its performance against the ETD-mr-SAV-MS2 scheme  \cite{WangWangZhang2026}, the BDF2-mr-SAV scheme \cite{han2025highly}, and the classical ETD-MS2 scheme \eqref{eqn:ETD-MS2}. In addition, we invesitgate the effective of the mean-reversion parameter $\gamma$.

\begin{exm}[Accuracy test]\label{exm:conv_test}
In this example, we consider the periodic two-dimensional Navier--Stokes equations on  $\Omega=(0,2\pi)^2$. We set the viscosity $\nu=10^{-3}$, the mean-reverting parameter $\gamma=1000$, and two different terminal times $T=0.1, 1$. The external force is set to be
\begin{equation*}
    f(x,y,t)=\cos(x).
\end{equation*}
The initial vorticity is prescribed by a smooth trigonometric field containing multiple spatial modes,
\begin{equation}\label{eqn:smooth_trig_vorticity}
    \omega_0(x,y) = \sum_{k,m=1}^{10} \frac{1}{(k^2+m^2)^{3/2}}\cos(kx)\cos(my).
\end{equation}
The initial value of the auxiliary variable is set to $q^0=1$. A Fourier spectral method with $256$ modes in each spatial direction is used for spatial discretization. Since the solution and the forcing are smooth and periodic, the spatial discretization error is expected to be negligible compared with the temporal discretization error. A reference solution is generated by the ETDRK4 scheme \cite{kassam2005fourth} with the uniform time step $\tau_{\rm ref}=0.1\times 2^{-10}$.
\end{exm}

We first consider the fixed-step case. The time step is set to  $\tau=0.1\times 2^{-k}$, with the tested values listed in Tables~ \ref{tab:fixed-step-convergence-cpu-1} and \ref{tab:fixed-step-convergence-cpu} for the two terminal times. Additional time steps sizes are considered in the regime when the classical ETD-MS2 scheme numerically blows up at $T=1$ in order to better understand the performance of the scheme. We compare the classical ETD-MS2 scheme, the BDF2-mr-SAV scheme, the ETD-mr-SAV-MS2 scheme, and the proposed ETD-mr-SAV-MS2-L scheme. The errors are measured at the final time $T = 0.1$ and $T=1$ using the discrete $L^2$ error of the vorticity, and the CPU time is also reported in order to compare accuracy and efficiency.

\begin{table}[htbp]
\centering
\caption{Vorticity $L^2$ errors, observed convergence rates, and CPU times for the tested schemes at $T=0.1$ under fixed time stepping. Here $\tau$ denotes the uniform time-step size.}
\label{tab:fixed-step-convergence-cpu-1}
\resizebox{\textwidth}{!}{%
\begin{tabular}{ccccccccccccc}
\toprule
$\tau$ & \multicolumn{3}{c}{ETD-MS2} & \multicolumn{3}{c}{BDF2-mr-SAV} & \multicolumn{3}{c}{ETD-mr-SAV-MS2} & \multicolumn{3}{c}{ETD-mr-SAV-MS2-L} \\
\cmidrule(lr){2-4} \cmidrule(lr){5-7} \cmidrule(lr){8-10} \cmidrule(lr){11-13}
 & Error & Order & CPU (s) & Error & Order & CPU (s) & Error & Order & CPU (s) & Error & Order & CPU (s) \\
\midrule
$0.1\times 2^{-2}$ & $2.332\times 10^{-6}$ & -- & 0.20 & $3.097\times 10^{-6}$ & -- & 0.14 & $2.332\times 10^{-6}$ & -- & 0.18 & $1.573\times 10^{-6}$ & -- & 0.28 \\
$0.1\times 2^{-3}$ & $6.713\times 10^{-7}$ & 1.80 & 0.29 & $9.939\times 10^{-7}$ & 1.64 & 0.18 & $6.713\times 10^{-7}$ & 1.80 & 0.29 & $5.797\times 10^{-7}$ & 1.44 & 0.43 \\
$0.1\times 2^{-3.2}$ & $5.064\times 10^{-7}$ & 2.03 & 0.34 & $7.596\times 10^{-7}$ & 1.94 & 0.20 & $5.064\times 10^{-7}$ & 2.03 & 0.34 & $4.461\times 10^{-7}$ & 1.89 & 0.50 \\
$0.1\times 2^{-3.3}$ & $4.446\times 10^{-7}$ & 1.88 & 0.34 & $6.726\times 10^{-7}$ & 1.76 & 0.20 & $4.446\times 10^{-7}$ & 1.88 & 0.34 & $3.957\times 10^{-7}$ & 1.73 & 0.50 \\
$0.1\times 2^{-3.4}$ & $3.846\times 10^{-7}$ & 2.09 & 0.36 & $5.855\times 10^{-7}$ & 2.00 & 0.21 & $3.846\times 10^{-7}$ & 2.09 & 0.37 & $3.449\times 10^{-7}$ & 1.98 & 0.53 \\
$0.1\times 2^{-3.5}$ & $3.400\times 10^{-7}$ & 1.78 & 0.39 & $5.181\times 10^{-7}$ & 1.76 & 0.22 & $3.400\times 10^{-7}$ & 1.78 & 0.39 & $3.078\times 10^{-7}$ & 1.64 & 0.57 \\
$0.1\times 2^{-3.6}$ & $3.021\times 10^{-7}$ & 1.71 & 0.41 & $4.608\times 10^{-7}$ & 1.69 & 0.23 & $3.021\times 10^{-7}$ & 1.71 & 0.42 & $2.759\times 10^{-7}$ & 1.58 & 0.61 \\
% $0.1\times 2^{-3}$ & $6.713\times 10^{-7}$ & -- & 0.30 & $9.939\times 10^{-7}$ & -- & 0.21 & $6.713\times 10^{-7}$ & -- & 0.29 & $5.797\times 10^{-7}$ & -- & 0.42 \\
% $0.1\times 2^{-3.2}$ & $5.064\times 10^{-7}$ & 2.03 & 0.34 & $7.596\times 10^{-7}$ & 1.94 & 0.23 & $5.064\times 10^{-7}$ & 2.03 & 0.38 & $4.461\times 10^{-7}$ & 1.89 & 0.49 \\
% $0.1\times 2^{-3.4}$ & $3.846\times 10^{-7}$ & 1.98 & 0.37 & $5.855\times 10^{-7}$ & 1.88 & 0.25 & $3.846\times 10^{-7}$ & 1.98 & 0.41 & $3.449\times 10^{-7}$ & 1.86 & 0.59 \\
% $0.1\times 2^{-3.6}$ & $3.021\times 10^{-7}$ & 1.74 & 0.42 & $4.608\times 10^{-7}$ & 1.73 & 0.27 & $3.021\times 10^{-7}$ & 1.74 & 0.43 & $2.759\times 10^{-7}$ & 1.61 & 0.67 \\
$0.1\times 2^{-3.8}$ & $2.323\times 10^{-7}$ & 1.89 & 0.46 & $3.573\times 10^{-7}$ & 1.84 & 0.28 & $2.323\times 10^{-7}$ & 1.89 & 0.45 & $2.150\times 10^{-7}$ & 1.80 & 0.65 \\
$0.1\times 2^{-4}$ & $1.791\times 10^{-7}$ & 1.88 & 0.52 & $2.765\times 10^{-7}$ & 1.85 & 0.29 & $1.791\times 10^{-7}$ & 1.88 & 0.51 & $1.676\times 10^{-7}$ & 1.80 & 0.80 \\
$0.1\times 2^{-5}$ & $4.622\times 10^{-8}$ & 1.95 & 0.90 & $7.268\times 10^{-8}$ & 1.93 & 0.46 & $4.622\times 10^{-8}$ & 1.95 & 0.91 & $4.476\times 10^{-8}$ & 1.91 & 1.41 \\
$0.1\times 2^{-6}$ & $1.174\times 10^{-8}$ & 1.98 & 1.73 & $1.862\times 10^{-8}$ & 1.96 & 0.85 & $1.174\times 10^{-8}$ & 1.98 & 1.78 & $1.155\times 10^{-8}$ & 1.95 & 2.46 \\
$0.1\times 2^{-7}$ & $2.958\times 10^{-9}$ & 1.99 & 3.30 & $4.713\times 10^{-9}$ & 1.98 & 1.50 & $2.958\times 10^{-9}$ & 1.99 & 3.47 & $2.932\times 10^{-9}$ & 1.98 & 4.78 \\
$0.1\times 2^{-8}$ & $7.424\times 10^{-10}$ & 1.99 & 6.50 & $1.185\times 10^{-9}$ & 1.99 & 3.25 & $7.424\times 10^{-10}$ & 1.99 & 6.86 & $7.369\times 10^{-10}$ & 1.99 & 9.49 \\
$0.1\times 2^{-9}$ & $1.860\times 10^{-10}$ & 2.00 & 13.06 & $2.972\times 10^{-10}$ & 2.00 & 5.95 & $1.860\times 10^{-10}$ & 2.00 & 13.34 & $1.829\times 10^{-10}$ & 2.01 & 18.83 \\
$0.1\times 2^{-10}$ & $4.654\times 10^{-11}$ & 2.00 & 26.40 & $7.442\times 10^{-11}$ & 2.00 & 12.13 & $4.654\times 10^{-11}$ & 2.00 & 26.75 & $4.388\times 10^{-11}$ & 2.06 & 40.82 \\
\bottomrule
\end{tabular}%
}
\end{table}

% Table \ref{tab:fixed-step-convergence-cpu-1} reports 

\begin{table}[htbp]
\centering
\caption{Vorticity $L^2$ errors, observed convergence rates, and CPU times for the tested schemes at $T=1$ under fixed time stepping. Here $\tau$ denotes the uniform time-step size.}
\label{tab:fixed-step-convergence-cpu}
\resizebox{\textwidth}{!}{%
\begin{tabular}{ccccccccccccc}
\toprule
 & \multicolumn{3}{c}{ETD-MS2} & \multicolumn{3}{c}{BDF2-mr-SAV} & \multicolumn{3}{c}{ETD-mr-SAV-MS2} & \multicolumn{3}{c}{ETD-mr-SAV-MS2-L} \\
\cmidrule(lr){2-4} \cmidrule(lr){5-7} \cmidrule(lr){8-10} \cmidrule(lr){11-13}
$\tau$ & Error & Order & CPU (s) & Error & Order & CPU (s) & Error & Order & CPU (s) & Error & Order & CPU (s) \\
\midrule
$0.1\times 2^{-2}$   & \texttt{NaN} & -- & 1.21 & $1.400\times 10^{-2}$ & -- & 0.70 & $3.068\times 10^{-3}$ & -- & 1.14 & $3.148\times 10^{-2}$ & -- & 1.83 \\
$0.1\times 2^{-3}$   & \texttt{NaN} & -- & 2.22 & $1.536\times 10^{-2}$ & \texttt{pre-asymp} & 1.33 & $3.127\times 10^{-3}$ & \texttt{pre-asymp} & 2.29 & $3.041\times 10^{-2}$ & \texttt{pre-asymp} & 3.37 \\
$0.1\times 2^{-3.2}$ & \texttt{NaN} & -- & 2.52 & $1.418\times 10^{-2}$ & \texttt{pre-asymp} & 1.58 & $3.114\times 10^{-3}$ & \texttt{pre-asymp} & 2.73 & $2.103\times 10^{-2}$ & \texttt{pre-asymp} & 4.06 \\
$0.1\times 2^{-3.3}$ & \texttt{NaN} & -- & 2.74 & $1.438\times 10^{-2}$ & \texttt{pre-asymp} & 1.65 & $3.235\times 10^{-3}$ & \texttt{pre-asymp} & 2.83 & $2.063\times 10^{-2}$ & \texttt{pre-asymp} & 4.36 \\
$0.1\times 2^{-3.4}$ & \texttt{NaN} & -- & 2.93 & $1.416\times 10^{-2}$ & \texttt{pre-asymp} & 1.61 & $3.200\times 10^{-3}$ & \texttt{pre-asymp} & 3.14 & $2.036\times 10^{-2}$ & \texttt{pre-asymp} & 4.58 \\
$0.1\times 2^{-3.5}$ & \texttt{NaN} & -- & 3.13 & $1.408\times 10^{-2}$ & \texttt{pre-asymp} & 1.87 & $2.886\times 10^{-3}$ & \texttt{pre-asymp} & 3.37 & $1.979\times 10^{-2}$ & \texttt{pre-asymp} & 5.38 \\
$0.1\times 2^{-3.6}$ & $3.453\times 10^{3}$ & -- & 3.24 & $1.402\times 10^{-2}$ & \texttt{pre-asymp} & 1.87 & $2.917\times 10^{-3}$ & \texttt{pre-asymp} & 3.64 & $1.985\times 10^{-2}$ & \texttt{pre-asymp} & 5.54 \\
$0.1\times 2^{-3.8}$ & $1.641\times 10^{-3}$ & \texttt{pre-asymp} & 3.11   & $1.038\times 10^{-2}$ & \texttt{pre-asymp} & 2.43   & $2.408\times 10^{-3}$ & \texttt{pre-asymp} & 3.98 & $2.764\times 10^{-3}$ & \texttt{pre-asymp} & 5.71 \\
$0.1\times 2^{-4}$   & $4.491\times 10^{-6}$ & \texttt{pre-asymp} & 3.80   & $4.480\times 10^{-5}$ & \texttt{pre-asymp} & 2.75   & $4.491\times 10^{-6}$ & \texttt{pre-asymp} & 4.41 & $4.475\times 10^{-6}$ & \texttt{pre-asymp} & 6.76 \\
$0.1\times 2^{-5}$   & $1.121\times 10^{-6}$ & 2.00               & 4.49   & $1.789\times 10^{-6}$ & \texttt{pre-asymp} & 5.44   & $1.121\times 10^{-6}$ & 2.00 & 6.20               & $1.118\times 10^{-6}$ & 2.00 & 8.05 \\
$0.1\times 2^{-6}$   & $2.803\times 10^{-7}$ & 2.00               & 9.21   & $4.480\times 10^{-7}$ & 2.00               & 10.55  & $2.803\times 10^{-7}$ & 2.00 & 14.39              & $2.799\times 10^{-7}$ & 2.00 & 15.66 \\
$0.1\times 2^{-7}$   & $7.013\times 10^{-8}$ & 2.00               & 18.30  & $1.121\times 10^{-7}$ & 2.00               & 21.16  & $7.013\times 10^{-8}$ & 2.00 & 27.14              & $7.008\times 10^{-8}$ & 2.00 & 31.23 \\
$0.1\times 2^{-8}$   & $1.754\times 10^{-8}$ & 2.00               & 35.76  & $2.806\times 10^{-8}$ & 2.00               & 41.83  & $1.754\times 10^{-8}$ & 2.00 & 52.66              & $1.752\times 10^{-8}$ & 2.00 & 62.68 \\
$0.1\times 2^{-9}$   & $4.386\times 10^{-9}$ & 2.00               & 74.76  & $7.017\times 10^{-9}$ & 2.00               & 86.28  & $4.386\times 10^{-9}$ & 2.00 & 90.05              & $4.356\times 10^{-9}$ & 2.01 & 125.14 \\
$0.1\times 2^{-10}$  & $1.097\times 10^{-9}$ & 2.00               & 166.69 & $1.755\times 10^{-9}$ & 2.00               & 168.78 & $1.097\times 10^{-9}$ & 2.00 & 177.25             & $1.064\times 10^{-9}$ & 2.03 & 256.84 \\
\bottomrule
\end{tabular}%
}
\end{table}

%Table~\ref{tab:fixed-step-convergence-cpu} presents the performance of four different schemes: ETD-mS2, ETD-mr-SAV-MS2, ETD-mr-SAV-MS2-L, and BDF2-mr-SAV \cite{han2025highly}. 
We have the following observations.
 (i) All the tested second-order schemes recover the expected second-order convergence in time after the coarse-step pre-asymptotic regime.
 (ii) Schemes reach second-order accuracy at larger time-step sizes at $T=0.1$ vs $T=1$. This suggests that the computation with larger terminal time might be more challenging.
 (iii) The classical ETD-MS2 scheme blows up numerically at $T=1$ for relatively large step sizes ($0.1\times 2^{-2}\ge \tau\ge 0.1\times 2^{-3.5}$). This is consistent with the lack of theoretical unconditional stability of the scheme.
 (iv) All the mr-SAV-based schemes produce numerical results with error of the order of $10^{-2}$ or smaller for the same relatively large steps sizes reported in (iii). This highlights the performance advantage of the mr-SAV-based schemes.
 (v) The classical ETD-MS2 scheme has essentially reached the theoretical 2nd-order convergence for steps sizes $\tau=0.1\times 2^{-3.2}, 0.1\times 2^{-3.4}$. Moreover, the relative error generated by the scheme at $T=0.1$ is on the order of $10^{-7}$ or smaller for the time step sizes tested. This is in sharp contrast to the numerical blow-up at $T=1$ for the same step sizes. Hence, even if a scheme produces a small error at a fixed terminal time and achieves the designed order of accuracy, its performance at later time could be problematic in the absence of theoretical long time stability guarantee.
 (vi) The mean-reverting schemes produce errors of the same order as the ETD-MS2 scheme at small time steps. Hence the mean-reverting SAV is able to stabilize the scheme without sacrifying the accuracy. 
 (vii) The ETD-mr-SAV schemes and the BDF2-mr-SAV scheme produce errors of the same order although the prefactors are smaller for the ETD schemes at small time steps. This is consistent with the intuition that the BDF2 scheme may be more diffusive than the ETD schemes.
 (viii) The CPU times are higher for the ETD-mr-SAV-MS2-L scheme due to the additional matrix-vector multiplications and inner products associated with the direct and inverse Laplace transform of the scalar auxiliary variable. Additional discussion about the computational complexities and related pros and cons can be found at the end of section 3.1. %illustrating the improved performance of the mean-reverting SAV treatment. As the time step is refined, the ETD-based mr-SAV schemes exhibit smaller error constants than the BDF2-mr-SAV scheme, and ETD-mr-SAV-MS2-L gives slightly smaller vorticity errors than ETD-mr-SAV-MS2 on the finest meshes. This improvement comes with a modest additional cost: both ETD-mr-SAV-MS2 and ETD-mr-SAV-MS2-L are dominated by FFT-type operations of order $O(N_s\log N_s)$ per time step, where $N_s$ denotes the number of spatial grid points, while ETD-mr-SAV-MS2-L additionally evaluates the auxiliary variable through a Laplace-inversion procedure. With a fixed number $M$ of Talbot nodes, this adds an $O(MN_s)$ contribution from frequency-space inner products, increasing the computational constant and explaining the slightly larger CPU times. Nevertheless, the Laplace-based formulation yields a linearly implicit auxiliary-variable update that is theoretically uniquely solvable, so ETD-mr-SAV-MS2-L trades a small loss in efficiency for a stronger solvability guarantee and slightly improved accuracy in this test.
 %We also observe similar phenomena at $T=0.1$ with the interesting difference that the error produced by the classical ETD-MS2 scheme is of the order of $10^{-7}$ for $\tau=0.1\times 2^{-3}, 0.1\times 2^{-3.2}, 0.1\times 2^{-2.4}$. This is in sharp constrast to the NaN at $T=1$.

Next, we examine the variable-step case. Starting from a uniform partition with $N$ time steps, we perturb each time step by a random relative amplitude of $15\%$ and subtract the mean perturbation so that the resulting time-step sequence still reaches the final time $T=1$. The total number of time steps is chosen as $N=10\times 2^k$, $k=4,5,\ldots,10$.

The variable-step convergence results for the ETD-mr-SAV-MS2-L scheme are reported in Table~\ref{tab:variable-step-etd-mrsav-ms2-l}. The observed rates remain close to two throughout the refinement sequence. This confirms that the proposed ETD-mr-SAV-MS2-L method preserves its second-order temporal accuracy under moderately perturbed time-step sequences, demonstrating the robustness of the method with respect to variable time stepping.

\begin{table}[htbp]
\centering
\caption{Vorticity $L^2$ errors and observed convergence rates computed by ETD-mr-SAV-MS2-L at $T=1$ under a $15\%$ perturbed variable-step sequence. Here $N$ denotes the total number of time steps.}
\label{tab:variable-step-etd-mrsav-ms2-l}
\begin{tabular}{@{} l *{7}{c} @{}}
\toprule
$N$ & $2^{4}\times 10$ & $2^{5}\times 10$ & $2^{6}\times 10$ & $2^{7}\times 10$ & $2^{8}\times 10$ & $2^{9}\times 10$ & $2^{10}\times 10$ \\
\midrule
Error & 4.500e-6 & 1.130e-6 & 2.850e-7 & 7.090e-8 & 1.770e-8 & 4.420e-9 & 1.080e-9 \\
Rate & -- & 1.99 & 1.99 & 2.01 & 2.00 & 2.00 & 2.03 \\
\bottomrule
\end{tabular}
\end{table}

We now demonstrate that the mean-reverting mechanism is not only theoretically important to our uniform-in-time bound, it can also enhance the performance by suppressing  numerical drifts. 
\begin{exm}\label{exm:mean_reverting_test}[Mean-reversion parameter effect]
In this example, we test the effect of the mean-reversion parameter $\gamma$. When $\gamma=0$, we recover the so-called ZEC scheme \cite{zhang2024unified} or the standard SAV scheme for NSE equations \cite{li2022new}. Notice that a positive mean-reversion parameter $\gamma$ is not only crucial to the uniform-in-time stability, it also enhances the computational result in this numerical experiment.
We consider the vorticity--streamfunction equation on the periodic domain $\Omega=(0,2\pi)^2$ with viscosity $\nu=1/40$. 
The external forcing is chosen as $f=4\cos(4y)$, corresponding to a Kolmogorov-type forcing. 
In order to test the effect of the mean-reverting mechanism away from a simple steady profile, 
the initial data are generated from the smooth isotropic perturbation streamfunction
\begin{equation*}
\psi_0(x,y)=\frac{1}{4} \sum_{k,m=1}^{10} \frac{1}{(k^2+m^2)^{3/2}}
% \sum_{\substack{\bm k\in\mathbb{Z}^2\\0<|\bm k|\le 10}}
% \frac{1}{|\bm k|^{3/2}}
\bigl(\cos(kx)+\sin(kx)\bigr)\bigl(\cos(my)+\sin(my)\bigr).
\end{equation*}
The spatial discretization uses a $256\times256$ Fourier grid. We compare the same second-order ETD-mr-SAV-MS2-L scheme with four mean-reverting parameters, $\gamma=1000,100,10,0$, under the time step $\tau=10^{-3}$. 
A reference solution is computed by the ETDRK4 scheme with the finer time step $\tau_{\rm ref}=5\times10^{-4}$.
\end{exm}

Figure~\ref{fig:mean_reverting_diagnostics} presents the relative $L^2$ error of the vorticity and the evolution of enstrophy for different choices of $\gamma$. 
Table~\ref{tab:mean_reverting_error_comparison} reports the corresponding relative errors at selected time instants. 
The comparison shows a clear stabilizing effect of the mean-reverting correction. When $\gamma$ is large, 
the numerical trajectory remains close to the ETDRK4 reference solution over the observed time interval. 
As $\gamma$ is reduced, this restoring effect becomes weaker, and the vorticity error increases accordingly. 
The case $\gamma=0$, which removes the mean-reverting correction, gives the largest deviation (more than $38\%$ at $t=20$) from the reference solution, 
while the intermediate choices $\gamma=100$ and $\gamma=10$ provide a gradual transition between the strongly mean-reverting and non-mean-reverting regimes.

The enstrophy diagnostics give a more nuanced message. Notice that the zero mean-reverting computation also tracks the reference enstrophy accurately in this test (relative error less than $0.5\%$ at $t=20$). 
Hence, the enstrophy error is not necessarily proportional to the full vorticity error. 
In particular, a computation may preserve an integral quantity such as enstrophy relatively well while still accumulating significant relative error in the vorticity field. Thus, enstrophy-related quantities are useful supplementary stability diagnostics, but they should not be used as standalone error indicators for adaptive time-stepping. Overall, these results support the role of the mean-reverting term as a practical stabilizing mechanism, not merely as a technical device in the theoretical analysis.

\begin{figure}[htbp]
\centering
\includegraphics[width=0.9\linewidth]{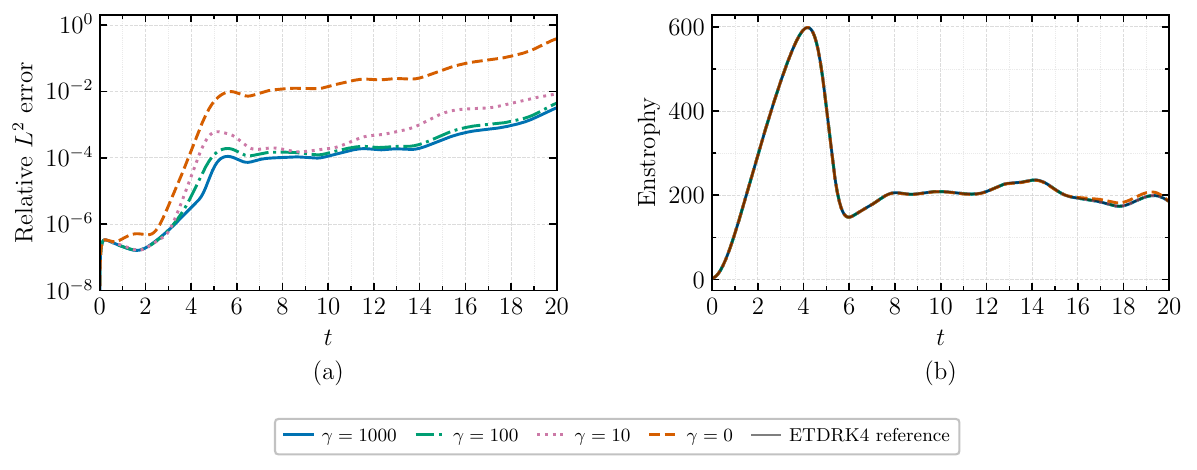}
\caption{Effect of the mean-reverting correction on the forced periodic vorticity problem. (a) Relative $L^2$ error of the vorticity with respect to the ETDRK4 reference solution.  (b) Enstrophy evolution for the reference solution and ETD-mr-SAV-MS2-L computations with $\gamma=1000,100,10,0$.}
\label{fig:mean_reverting_diagnostics}
\end{figure}

\begin{table}[htbp]
\centering
\caption{Comparison of relative vorticity errors $e_\omega$ and enstrophy errors $e_{\mathcal{E}}$ at selected times under different mean-reverting parameters. All errors are measured relative to the ETDRK4 reference solution.}
\label{tab:mean_reverting_error_comparison}
\resizebox{0.9\textwidth}{!}{%
\begin{tabular}{ccccccccc}
\toprule
\multirow{2}{*}{$t$} & \multicolumn{2}{c}{$\gamma=1000$} & \multicolumn{2}{c}{$\gamma=100$} & \multicolumn{2}{c}{$\gamma=10$} & \multicolumn{2}{c}{$\gamma=0$} \\
\cmidrule(lr){2-3} \cmidrule(lr){4-5} \cmidrule(lr){6-7} \cmidrule(lr){8-9}
 & $e_\omega$ & $e_{\mathcal{E}}$ & $e_\omega$ & $e_{\mathcal{E}}$ & $e_\omega$ & $e_{\mathcal{E}}$ & $e_\omega$ & $e_{\mathcal{E}}$ \\
\midrule
5  & $\mathbf{5.06\times 10^{-5}}$ & $5.36\times 10^{-5}$          & $1.22\times 10^{-4}$ & $\mathbf{1.93\times 10^{-5}}$ & $5.85\times 10^{-4}$ & $2.70\times 10^{-4}$ & $4.92\times 10^{-3}$ & $1.87\times 10^{-3}$ \\
10 & $\mathbf{1.12\times 10^{-4}}$ & $\mathbf{1.40\times 10^{-5}}$ & $1.38\times 10^{-4}$ & $1.48\times 10^{-5}$          & $1.86\times 10^{-4}$ & $1.41\times 10^{-5}$ & $1.40\times 10^{-2}$ & $1.58\times 10^{-3}$ \\
15 & $\mathbf{3.49\times 10^{-4}}$ & $\mathbf{9.20\times 10^{-6}}$ & $4.88\times 10^{-4}$ & $2.32\times 10^{-5}$          & $2.15\times 10^{-3}$ & $2.75\times 10^{-4}$ & $4.38\times 10^{-2}$ & $3.44\times 10^{-3}$ \\
20 & $\mathbf{3.19\times 10^{-3}}$ & $1.99\times 10^{-4}$          & $4.41\times 10^{-3}$ & $\mathbf{2.58\times 10^{-5}}$ & $8.43\times 10^{-3}$ & $2.02\times 10^{-3}$ & $3.88\times 10^{-1}$ & $4.86\times 10^{-3}$ \\
\bottomrule
\end{tabular}%
}
\end{table}

\subsection{Long-time stability}\label{longtimestability}

% \subsection{Kolmogorov flow}
\begin{exm}\label{exm:kolmogorov}[long time stability]
In this example, we simulate the Kolmogorov flow to compare the long-time performance of different schemes. We consider the two-dimensional periodic NSE on $(0,2\pi)^2$ in vorticity--streamfunction form \eqref{eqn:vs_ns}, with external forcing
\begin{equation}
f(x,y;t)=m\cos(my),
\end{equation}
following \cite{armbruster1996symmetries}. The corresponding steady-state solution (the basic Kolmogorov flow) is given by the streamfunction $\psi(x,y) = -\frac{1}{\nu m^3}\cos(my)$. 
%We define the Reynolds number as $\R = \frac{L v_{\text{avg}}}{\nu}$, where $L=2\pi$ is the domain length and $v_{\text{avg}}$ is the averaged velocity
%\[
%v_{\text{avg}} = \frac{1}{TL}\sqrt{\int_0^{T}\Vert \omega(t) \Vert^2 \, dt},
%\]
%with $T$ denoting the averaging time interval. The basic flow is stable for sufficiently small $\R$; when $\R$ exceeds a critical value, the flow loses stability and nontrivial time-dependent structures may emerge.
\end{exm}

We first investigate the long-time stability of the Kolmogorov forcing with $m=2$, $\nu = 1/20$, and $\gamma = 1000$. The initial streamfunction is taken as a perturbation of the steady state presented above, given by
\begin{equation*}
\psi_0(x,y) = -\frac{1}{\nu m^3}\cos(my) + 0.001 \cos(mx)\cos(my).
\end{equation*}

Figure~\ref{fig:long_time_stability_etdms} presents the evolution of the \(L^2\) norm of vorticity computed up to the final time \(T = 1000\) using the ETD-mr-SAV-MS2-L scheme with uniform time steps \(\tau=0.01,\,0.005,\,0.0025\). Over the full simulation interval, the \(L^2\) norm remains uniformly bounded, which is consistent with the theoretical stability results established in Section~3. When the time-step is not small, say $\tau=0.01$, the long-time statistics of the palinstrophy (one-half of the $L^2$ norm of the gradient of the vorticity) seems far-off from the behavior for smaller time-steps although it is uniformly bounded in time.

\begin{figure}[htbp]
\centering
\includegraphics[width=\linewidth]{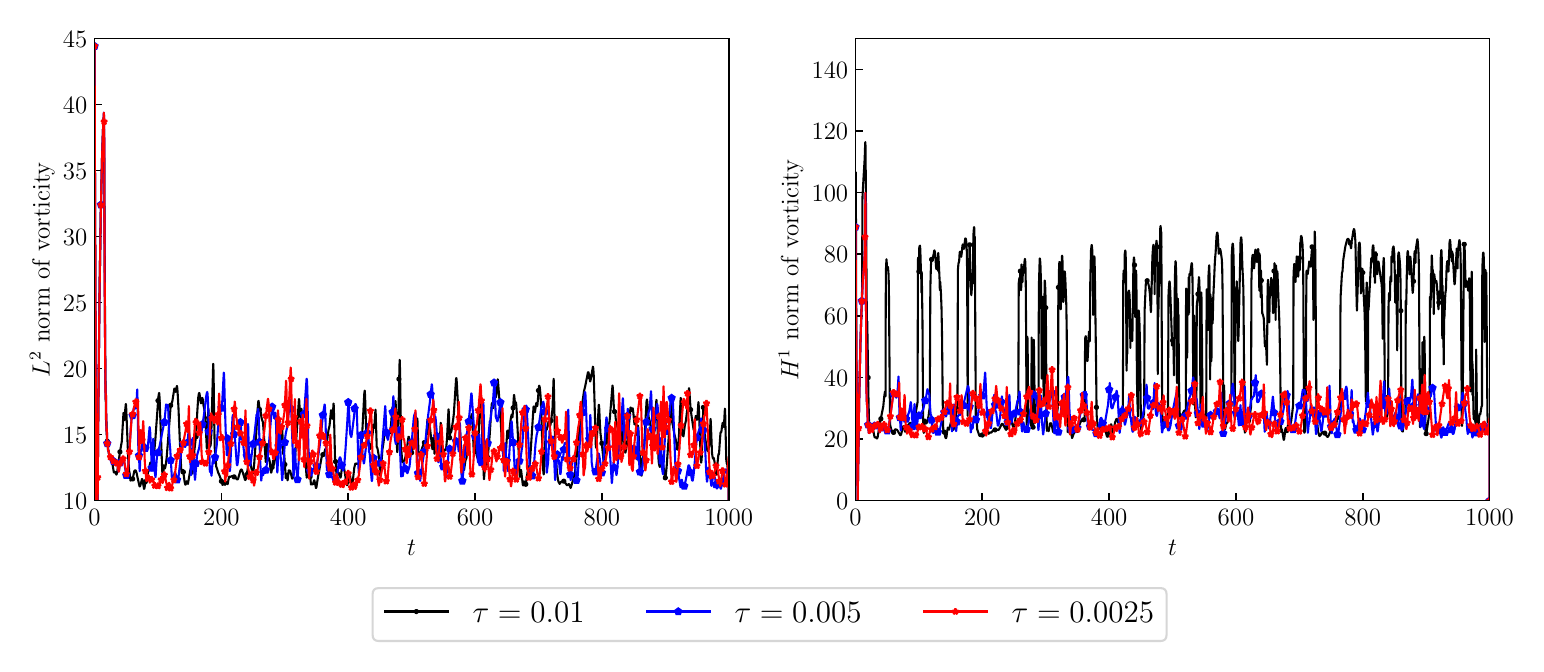}
\caption{Evolution of the vorticity $L^2$ norm computed by the ETD-mr-SAV-MS2-L and ETD-mr-SAV-MS12-L schemes for $\nu=1/20$ and $m=2$.}
\label{fig:long_time_stability_etdms}
\end{figure}

We next simulate the Kolmogorov flow with parameters $m=4$, $\nu = 1/40$ and $\gamma = 1000$. The initial streamfunction is taken as a perturbation of the steady state presented above,
\begin{equation}
\psi_0(x,y) = -\frac{1}{\nu m^3}\cos(my) + 0.001 \cos(mx)\cos(my).
\end{equation}
The final time is set to $T=10000$. We employ a fixed time step $\tau = 5 \times 10^{-4}$ and adopt the ETD-mr-SAV-MS2-L scheme and the ETD-mr-SAV-MS2 scheme proposed in \cite{WangWangZhang2026} for comparison. %Spatial discretization uses $256$ Fourier modes.

\ignore{
\begin{table}[htbp]
\centering
\caption{Reynolds numbers obtained by ETD-mr-SAV-MS2, ETD-mr-gSAV-MS2, and Ada ETD-mr-gSAV-MS2 for $\nu=1/40$ and $m=4$.}
\label{tab:rey_num}
\begin{tabular}{|c|c|c|c|}
\hline
Method & ETD-mr-SAV-ms2 & ETD-mr-gSAV-ms2 & Ada ETD-mr-gSAV-ms2 \\ \hline
Re & 43.52 & 43.43 & 43.61 \\ \hline
\end{tabular}
\end{table}

We first compute the Reynolds numbers associated with the trajectories generated by different schemes; see Table~\ref{tab:rey_num}. The results are consistent across methods. Moreover, the solution exhibits temporal intermittency/bursting behavior, manifested as occasional spikes in enstrophy, even at this moderate Reynolds number, consistent with earlier observations in \cite{armbruster1996symmetries,han2025highly}.
}

The time evolution of enstrophy (one-half of the $L^2$ norm of the vorticity squred) and palinstrophy (one-half of the $L^2$ norm of the gradient of the vorticity squared), calculated by the ETD-mr-SAV-MS2-L method and the ETD-mr-SAV-MS2 method, are displayed in Figure~\ref{fig:L2normevol} and Figure~\ref{fig:gL2normevol}, respectively.
Although individual trajectories diverge over long time, they remain uniformly bounded for all time, which further illustrates long-time stability. The separation of trajectories over long times is likely due to the intrinsic instability of the dynamics (e.g., positive Lyapunov exponents).

\begin{figure}[htbp]
\centering
\includegraphics[width=\linewidth]{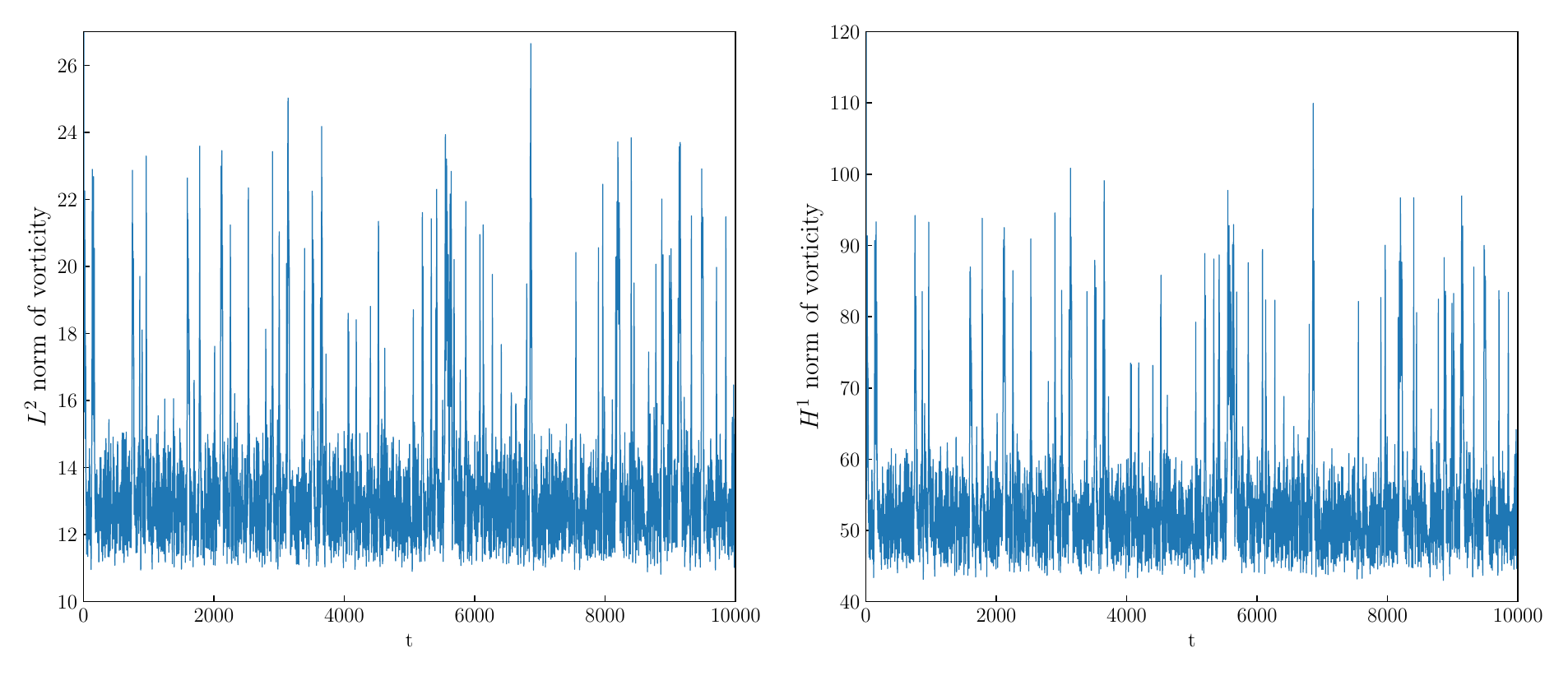}
\caption{Enstrophy (one-half of the $L^2$ norm of the vorticity squared) and palinstrophy (one-half of the $L^2$ norm of the gradient of the vorticity squared) as functions of time, computed using ETD-mr-SAV-MS2-L scheme with $\nu=1/40$, $m=4$, and $\tau=5\times 10^{-4}$.}
\label{fig:L2normevol}
\end{figure}

\begin{figure}[htbp]
\centering
\includegraphics[width=\linewidth]{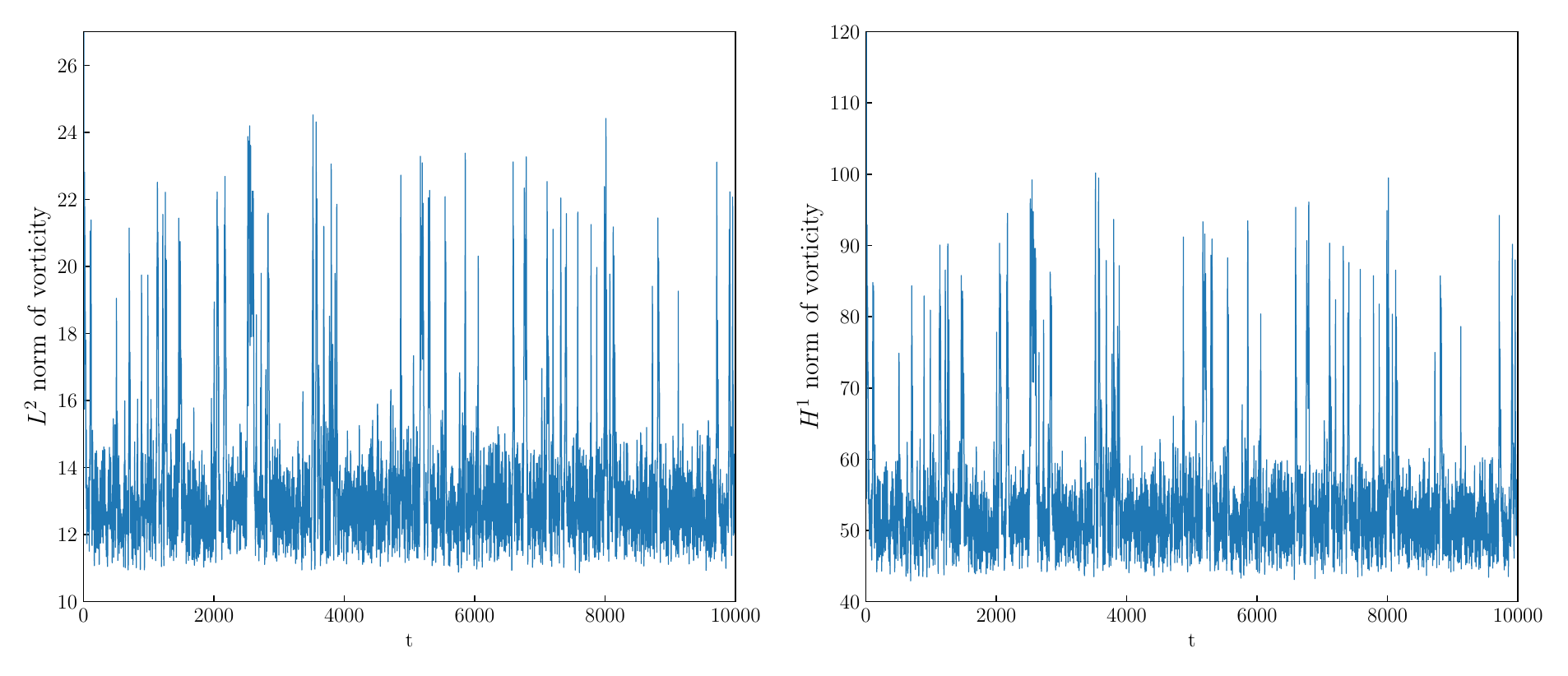}
\caption{Enstrophy (one-half of the $L^2$ norm of the vorticity squared) and palinstrophy (one-half of the $L^2$ norm of the gradient of the vorticity squared) as function of time, computed using ETD-mr-SAV-MS2 scheme with $\nu=1/40$, $m=4$, and $\tau=5\times 10^{-4}$.}
\label{fig:gL2normevol}
\end{figure}

\begin{figure}[htbp]
\centering
\includegraphics[width=0.6\linewidth]{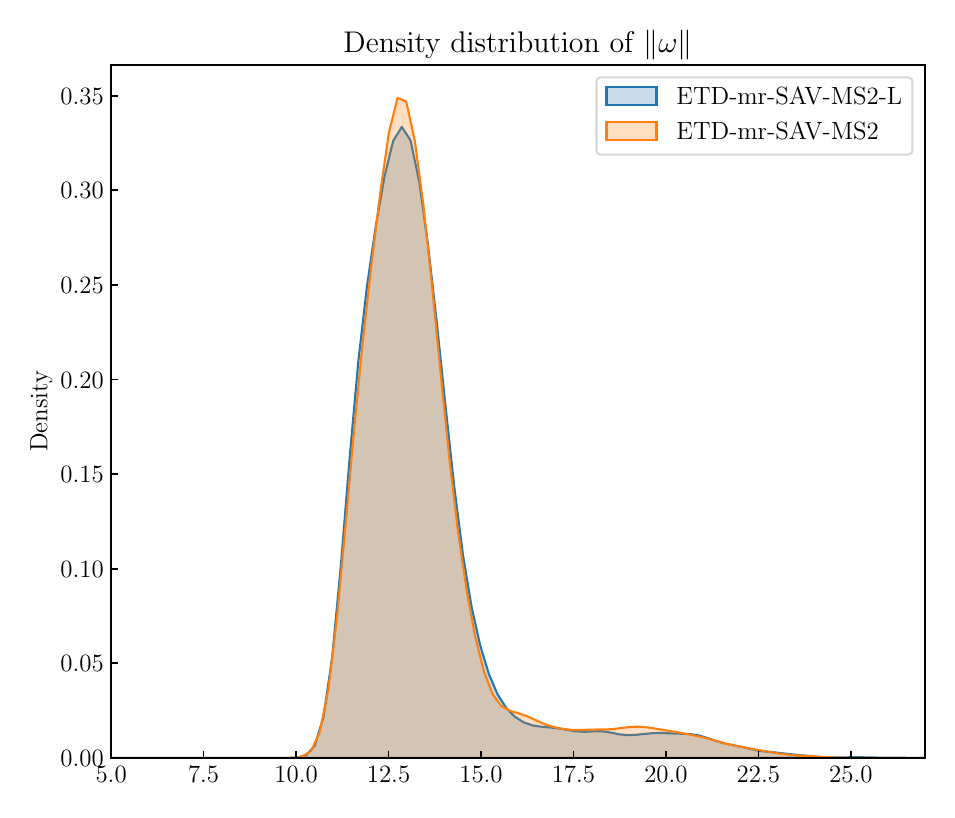}
\caption{Probability density functions of enstrophy obtained by ETD-mr-SAV-MS2-L, ETD-mr-SAV-MS2 for $\nu=1/40$, $m=4$, and $\tau=5\times 10^{-4}$.}
\label{fig:l2_norm_dist_2}
\end{figure}

To assess long-time statistical consistency, we consider the probability density function (PDF) of the enstrophy and compute the total variation distance between the enstrophy distributions obtained from different schemes using the same uniform temporal sampling. The PDFs are nearly indistinguishable (Figure~\ref{fig:l2_norm_dist_2}), and the total variation distances are small (Table~\ref{tab:tv_comparison_2}).
These observations provide evidence that these two schemes produce statistically consistent long-term behavior and are able to capture long-term statistics of the 2D NSE irrespective of the separation of trajectories \cite{foias2001navier,majda2006nonlinear,wang2012efficient}.
The fat tail in PDF is consistent with the observed bursting and intermittent behavior.

\begin{table}[htbp]
\centering
\caption{Total variation distance between enstrophy distributions obtained by ETD-mr-SAV-MS2-L and ETD-mr-SAV-MS2 for $\nu=1/40$, $m=4$, and $\tau=5\times 10^{-4}$.}
\label{tab:tv_comparison_2}
\begin{tabular}{@{}ccc@{}}
\toprule
\multicolumn{1}{c}{$L^2$ Norm of Vorticity} & \multicolumn{1}{c}{Total Variation Distance}  & \multicolumn{1}{c}{Relative $L^1$ Error} \\
\midrule
$ \Vert \omega \Vert < 17.5$  & $0.0199218$ & $4.243069\%$ \\
$ \Vert \omega \Vert \geq 17.5 $  & $0.0038525$ & $12.647748\%$ \\
$ -\infty \le \Vert \omega \Vert \le \infty$  & $0.0237743$ & $4.755107\%$\\
\bottomrule
\end{tabular}
\end{table}

\subsection{Adaptive Strategies}
In this section, we examine the performance of the proposed adaptive time-step method and the VSVO method separately.

\subsubsection{Adaptive Time-Step Method Test}

To facilitate a fair comparison with the fixed-step-size scheme, we adopt in this numerical experiment the identical external force and parameter configurations as those used in Example \ref{exm:kolmogorov}. We first investigate the stability performance of the adaptive time-step ETD-mr-SAV-MS12-L scheme under Kolmogorov forcing with \(m = 2\), \(\nu = 1/20\), and \(\gamma = 1000\), where the adaptive parameters are set as
\begin{equation}\label{adaptive_set}
\rho = 0.9,\quad \text{tol}_{\omega} = \text{tol}_{q} = 10^{-4},\quad\tau_{\min} = 1\times 10^{-5},\quad\tau_{\max} = 1\times10^{-2}.
\end{equation}
The initial stream function is set to be the same perturbation of the steady state as in Example \ref{exm:kolmogorov}, given by
\begin{equation*}
\psi_0(x,y) = -\frac{1}{\nu m^3}\cos(my) + 0.001 \cos(mx)\cos(my).
\end{equation*}

Figure~\ref{fig:long_time_stability_adp} shows the evolution of the \(L^2\)-norm and \(H^1\)-norm of the vorticity up to the final time \(T = 1000\), computed using the adaptive ETD-mr-SAV-MS12-L scheme and the ETD-mr-SAV-MS2-L scheme with a uniform time step \(\tau = 0.0025\). It is observed that, in comparison with the fixed time-step scheme, the adaptive algorithm can equally guarantee the uniform boundedness of the $L^2$-norm. This demonstrates that the theoretical stability results established in Section~3 remain valid for the case of variable time steps.

\begin{figure}[htbp]
\centering
\includegraphics[width=\linewidth]{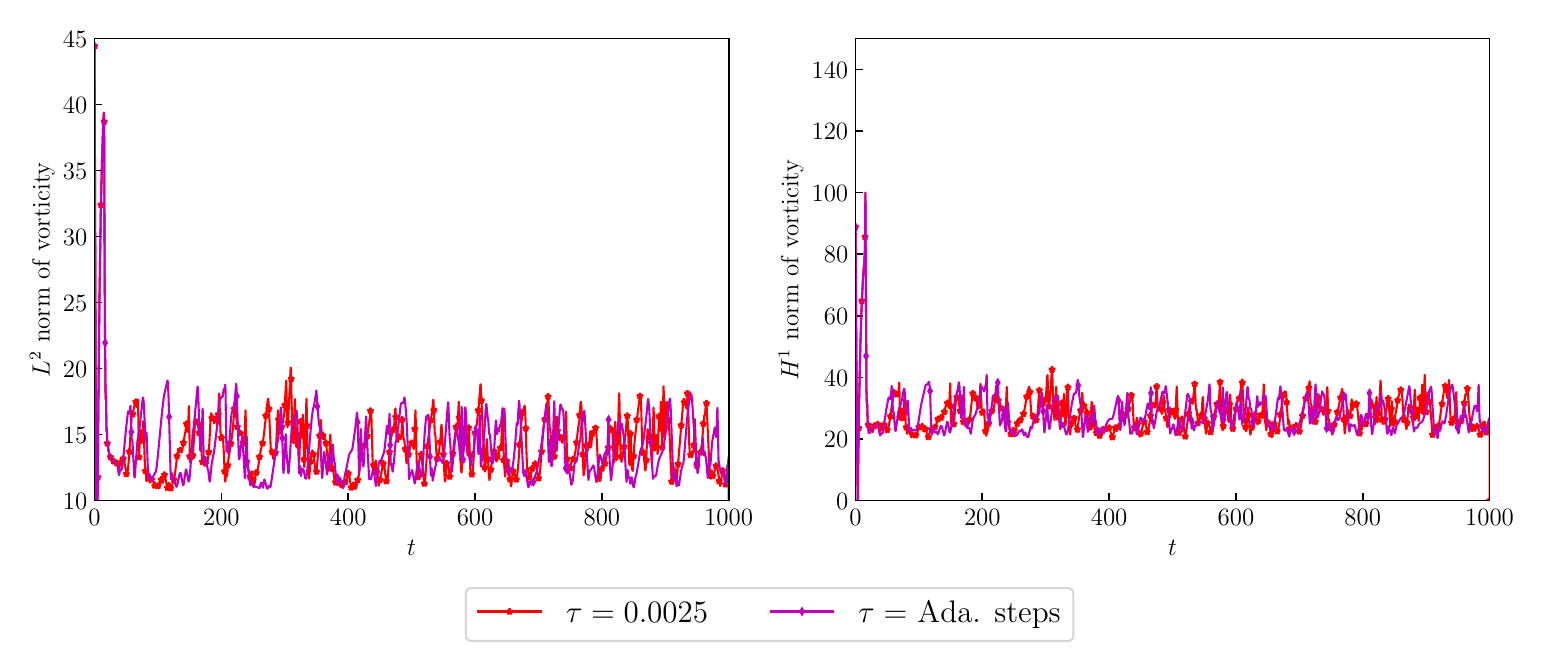}
\caption{Evolution of the vorticity $L^2$ norm computed by the ETD-mr-SAV-MS2-L and ETD-mr-SAV-MS12-L schemes for $\nu=1/20$ and $m=2$.}
\label{fig:long_time_stability_adp}
\end{figure}

To assess the short-time performance, we adopt the parameter settings from the first Kolmogorov test with $m=2$ and $\nu=1/20$, and choose the initial data as the flow profile at a random time between $20$ and $60$ from that simulation so that the trajectory starts near the attractor. Figure~\ref{fig:long_time_adap} illustrates the time step size and the corresponding errors (computed via the ETD-mr-SAV-MS12 scheme) over the time interval \( t \in [0,50] \), both presented as functions of time. Specifically, Figures~\ref{fig:long_time_adap} (a), (b), and (c) display the adaptive time step size, along with the absolute error in the scalar auxiliary variable and the relative $L^2$ errors of vorticity, which are used to control the time step. It can be observed that the adaptive algorithm efficiently adjusts the time step to ensure the errors remain below the prescribed tolerances.

To further characterize the actual evolution of errors, the ETDRK4 scheme with a finer time step is employed to produce a reference solution. The relative $ L^2 $-errors (using the reference solution as the numerical truth) are presented in Figure~\ref{fig:long_time_adap} (d). Even at the final simulation time $ t=50 $, these errors remain well-below $ 10^{-3} $ despite the accumulation of local errors; this demonstrates the effectiveness of the embedded adaptive pair ETD-mr-SAV-MS12-L scheme. 

\begin{figure}[htbp]
    \centering
    \includegraphics[width=0.8\linewidth]{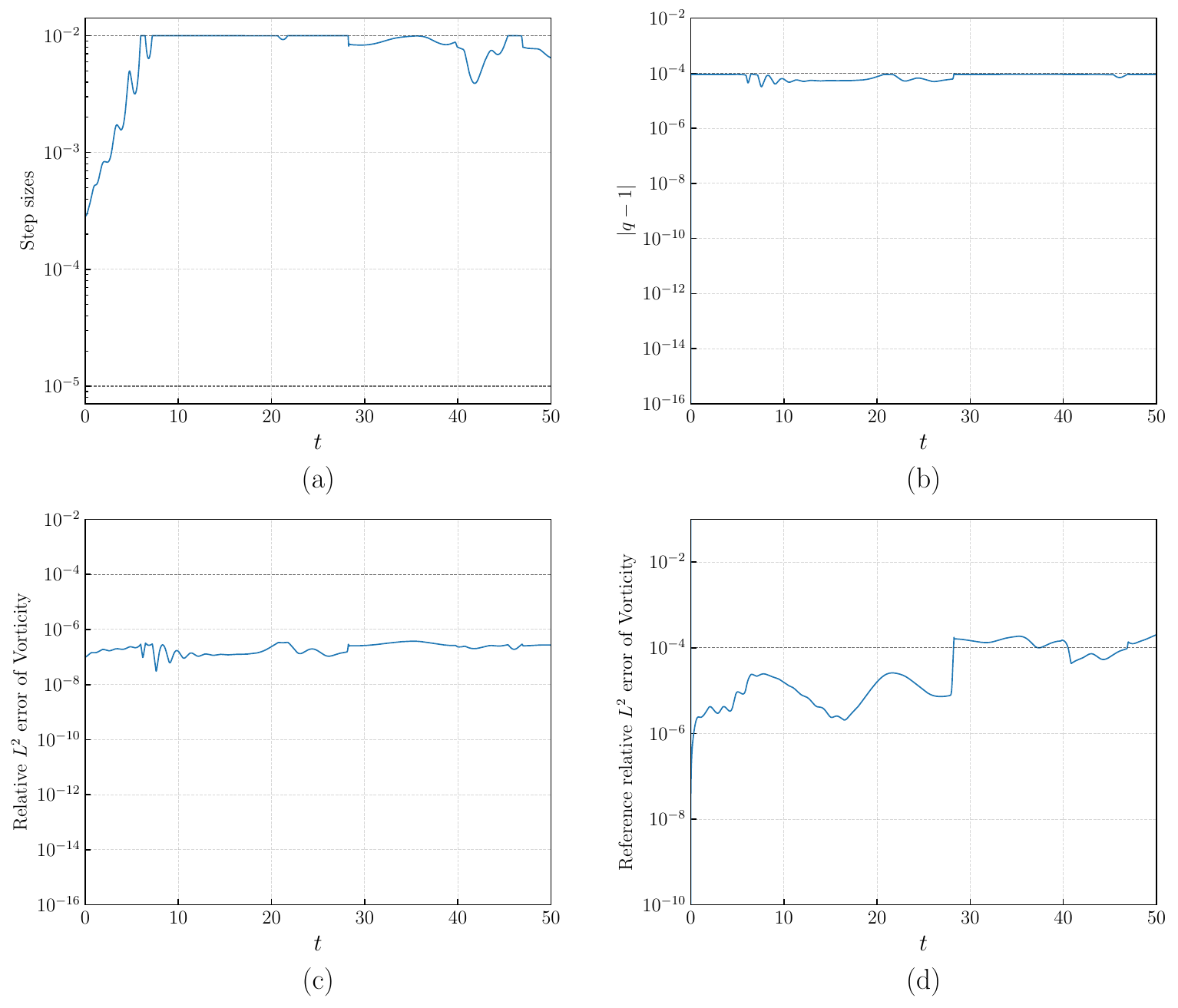}
    \caption{(a) Adaptive time step sizes as a function of time. (b) Absolute value of the auxiliary varialbe as a function of time. (c) Relative $L^2$ error of vorticity computed in the ETD-mr-SAV-MS12-L scheme and (d) Reference relative $L^2$ error computed via the ETD-mr-SAV-MS12-L scheme and ETDRK4 scheme.}
    \label{fig:long_time_adap}
\end{figure}

Next, we adopt the parameter settings from the second Kolmogorov test in Example \ref{exm:kolmogorov}, i.e., \(m=4\) and \(\nu=1/40\), together with the same initial streamfunction and spatial discretization configuration, to examine the very long time (up to $T=10^4$) performance of the adaptive time-stepping scheme.

The time evolution of the enstrophy and palinstrophy, computed using the ETD-mr-SAV-MS12-L method up to \(T = 10000\), is displayed in Figure \eqref{fig:L2normevol_adp}. It is observed that the trajectories remain uniformly bounded over long-time simulation, which further demonstrates the long-time stability of the proposed schemes. 
We also considered the probability density function (PDF) of the enstrophy and calculated the total variation distance between the enstrophy distributions obtained from the fixed time-step ETD-mr-SAV-MS2 scheme with $\tau = 5\times10^{-4}$ using the same uniform time sampling. 

As shown in Figure \ref{fig:l2_norm_dist_adp}, the PDFs obtained using the adaptive step-size and fixed step-size strategies are consistent. It can also be found in Table \ref{tab:tv_comparison} that the total variation distance is  small. This indicates that in long-term simulations, the adaptive strategy shows statistical consistency with the fixed step-size, demonstrating the effectiveness of the adaptive method.

\begin{figure}[htbp]
\centering
\includegraphics[width=\linewidth]{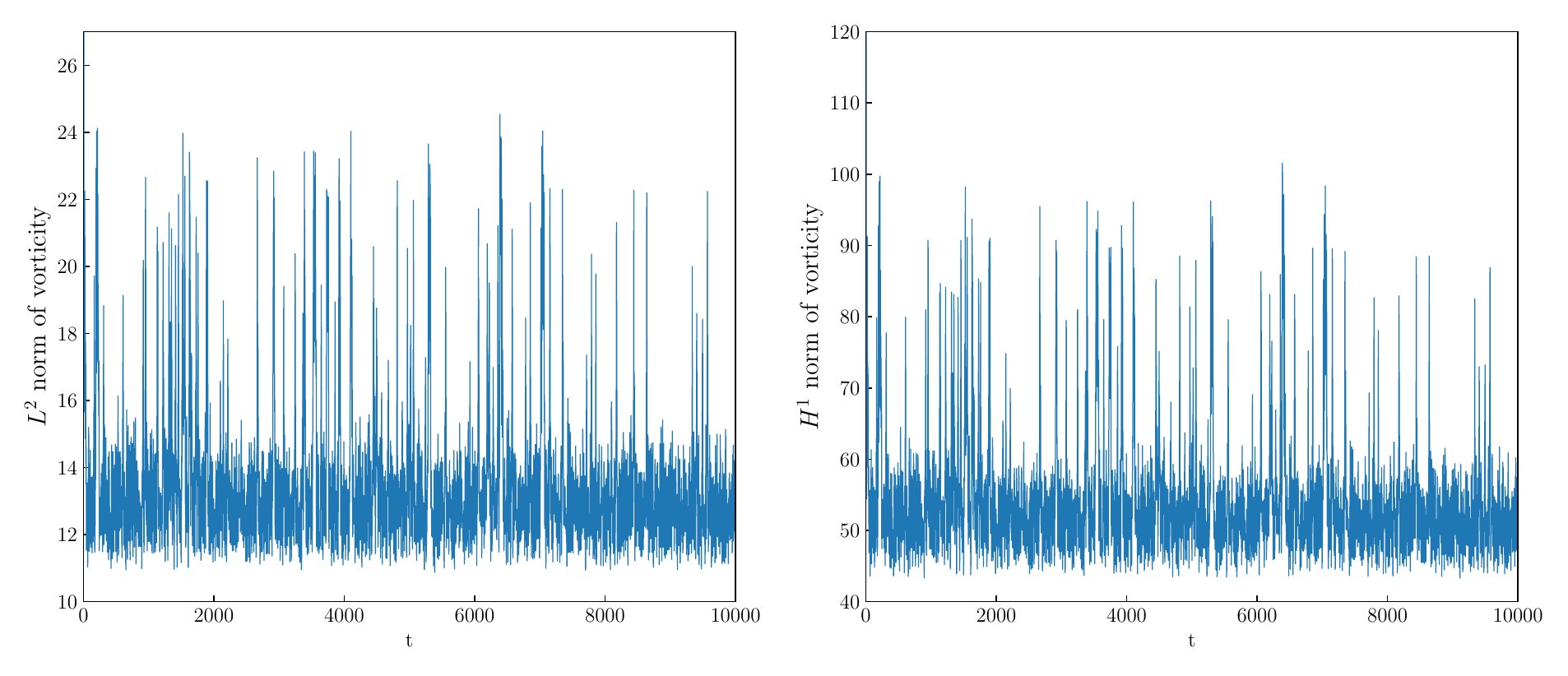}
\caption{Enstrophy ($L^2$ norm of vorticity) and palinstrophy ($H^1$ norm of vorticity) as function of time, computed using ETD-mr-SAV-MS12-L scheme with $\nu=1/40$ and $m=4$. }
\label{fig:L2normevol_adp}
\end{figure}

\begin{figure}[htbp]
\centering
\includegraphics[width=0.6\linewidth]{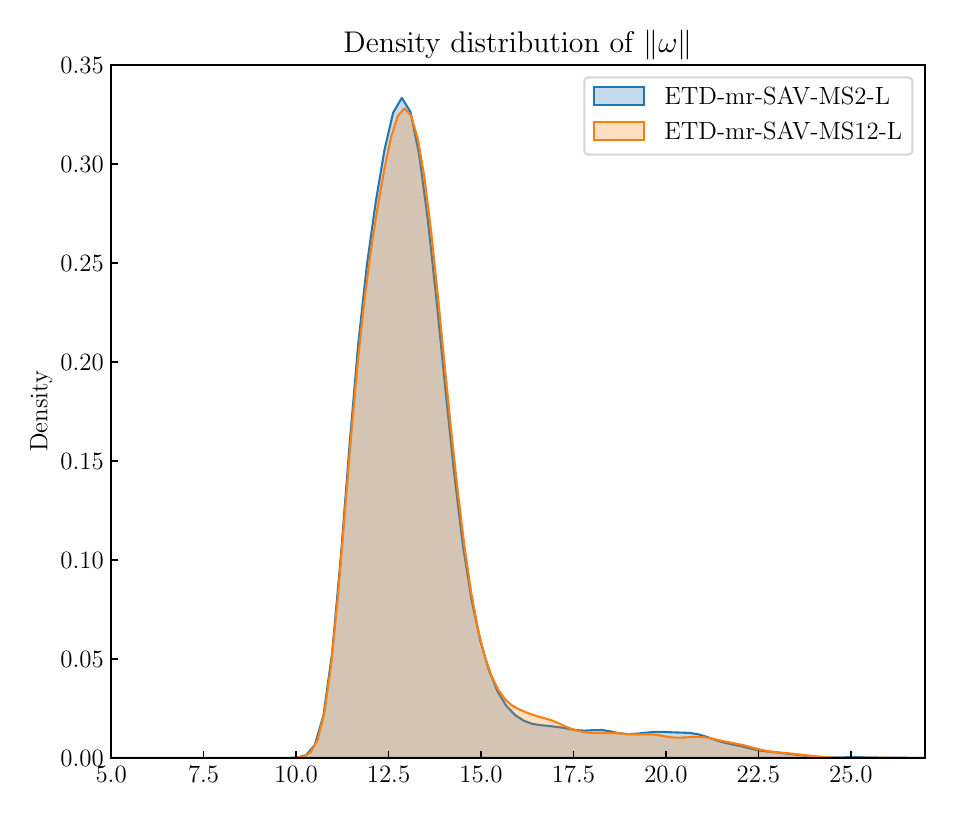}
\caption{Probability density functions of enstrophy obtained by ETD-mr-SAV-MS2-L, ETD-mr-SAV-MS12-L for $\nu=1/40$ and $m=4$.}
\label{fig:l2_norm_dist_adp}
\end{figure}

\begin{table}[htbp]
\centering
\caption{Total variation distance between enstrophy distributions obtained by ETD-mr-SAV-MS2-L, ETD-mr-SAV-MS12-L for $\nu=1/40$ and $m=4$.}
\label{tab:tv_comparison}
\begin{tabular}{@{}ccc@{}}
\toprule
\multicolumn{1}{c}{$L^2$ Norm of Vorticity} & \multicolumn{1}{c}{Total Variation Distance}  & \multicolumn{1}{c}{Relative $L^1$ Error} \\
\midrule
$ \Vert \omega \Vert < 17.5$  & $0.0156132$ & $3.325381\%$ \\
$ \Vert \omega \Vert \geq 17.5 $  & $0.0029722$ & $9.835980\%$ \\
$ -\infty \le \Vert \omega \Vert \le \infty$  & $0.0185854$ & $3.718358\%$\\
\bottomrule
\end{tabular}
\end{table}

Furthermore, we provide additional evidence of the advantages of adaptive time stepping, as illustrated in Figures~\ref{fig:adaptive_step_k}. Specifically, Figure~\ref{fig:adaptive_step_k} (a), (b), and (c) shows that the adaptive scheme automatically selects smaller time steps when bursting events occur (corresponding to spikes in the enstrophy), improving temporal resolution during rapid transients. In addition, to illustrate the correlation between the time step size and the enstrophy, we calculated their correlation coefficient (PCC). The mathematical expression of PCC is given by 
$$
PCC(x,y) = \frac{\sum (x - m_x) (y - m_y)}{\sqrt{\sum (x - m_x)^2 \sum (y - m_y)^2}}.
$$ 
where $m_x$, $m_y$ are the mean values of $x$ and $y$, respectively. The computed PCC value is -0.8123, % between the time step size and enstrophy,
suggesting a strong negative correlation. Specifically, as the enstrophy increases, the time step size decreases accordingly, this property is clearly reflected in Figure~\ref{fig:adaptive_step_k} (c). Figure~\ref{fig:adaptive_step_k} (d) further quantifies efficiency: the total number of time steps required by the adaptive algorithm is approximately one-sixth of that required by the fixed-step simulation. This demonstrates that the adaptive strategy achieves a favorable balance between accuracy and computational cost.

\begin{figure}[htbp]
    \centering
    \includegraphics[width=\linewidth]{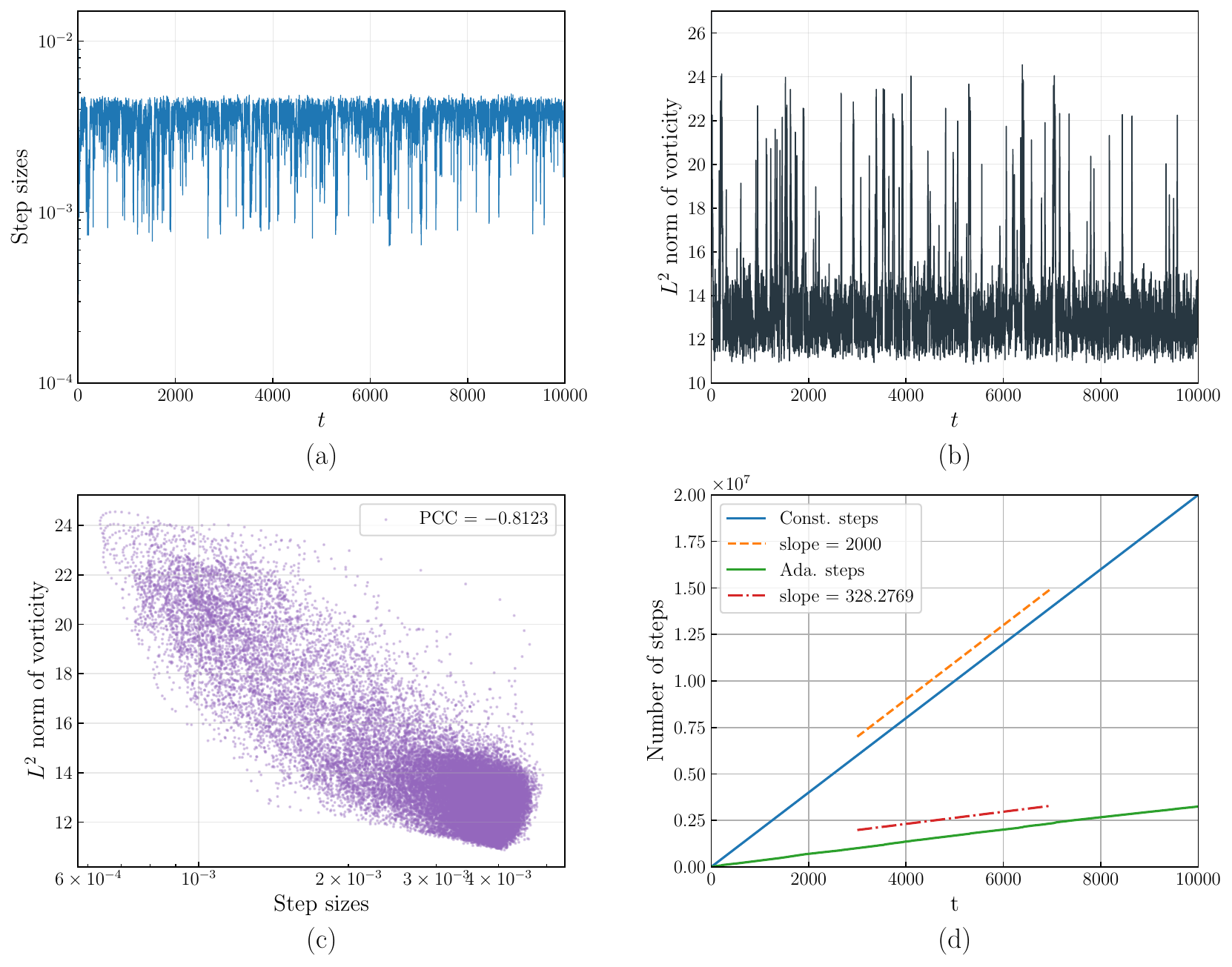}
    \caption{(a) Adaptive time step sizes as a function of time. (b) Enstrophy ($L^2$ norm of vorticity) as a function of time. (c) Correlation between step sizes and enstrophy ($L^2$ norm of vorticity) (PCC = $-0.8123$). (d) Number of steps over time for constant steps and adaptive steps.}
    \label{fig:adaptive_step_k}
\end{figure}

\subsubsection{VSVO Method Test}

Here we assess the performance of the adaptive VSVO method on a problem with multiple time scales, and demonstrate the superiority of the adaptive VSVO strategy in this case. We follow the classical stepsize-order control strategy described in \cite{li2020adaptive,decaria2021embedded,decaria2021variable}.

\begin{exm}\label{exm:adaptive}
In this experiment, we adopt the vorticity--streamfunction formulation \eqref{eqn:vs_ns} on the periodic domain $(0,2\pi)^2$. The external forcing, related to the classical Taylor-Green vortex problem, is multi-scale in time, and takes the form 
\begin{equation}
f(x,y,t)= -(2\nu F(t) + F'(t)) \cos(x) \cos(y)
\end{equation}
where $F(t)$ is chosen as the classical square-wave function in the form
\begin{equation}
F(t)= -\tanh \left(100 \sin\left(\frac{\pi}{5}\left(t + \frac{5}{2}\right)\right)\right) + 2.
\end{equation}
The corresponding exact solution is given by 
\begin{equation}
    \omega(x,y;t) = -2 F(t) \cos(x)\cos(y)
\end{equation}
which also exhibits the slow--fast temporal behavior characteristic of the forcing term.
\end{exm}

In the numerical simulation, we set $\nu = 0.05$ and $\gamma = 1000$, and employ $256$ Fourier modes for spatial discretizations. The VSVO scheme is adopted in the simulation, with its parameters specified as
\begin{equation}\label{eqn:vsvo_par}
    \text{tol} = 10^{-4},\quad \tau_{\min} = 10^{-5},\quad \tau_{\max} = 2\times 10^{-2}.
\end{equation}
For comparison, we also implement the ETD-mr-SAV-MS2-L scheme with a constant time step $\tau = 10^{-2}$.

Figure~\ref{fig:VSVO_1} displays the $L^2$ norm of vorticity and the $L^2$ error of vorticity as functions of time. Meanwhile, the adaptive time steps and adaptive order selections are presented in Figure~\ref{fig:VSVO_2}. The results demonstrate that the two numerical strategies yield nearly identical evolutions for the $L^2$ norm of vorticity. However, the fixed-step ETD-mr-SAV-MS2-L would incur an error of the order of $10^{-1}$ near the sharp temporal transition of the solution while the  VSVO method maintains an error of the order of $10^{-7}$ or less for all time. This demonstrates the superior performance of the VSVO strategy in capturing the transition behaviors. 
%More specifically, when the external force is nearly constant in time, the adaptive time step remains close to the maximum allowable value $\tau_{\max}=2\times 10^{-2}$, and the scheme employs the first-order solution for evolution. When the external force varies rapidly, the time steps decrease automatically, and the second-order solution is adopted.

\begin{figure}[htbp]
    \centering
    \includegraphics[width=\linewidth]{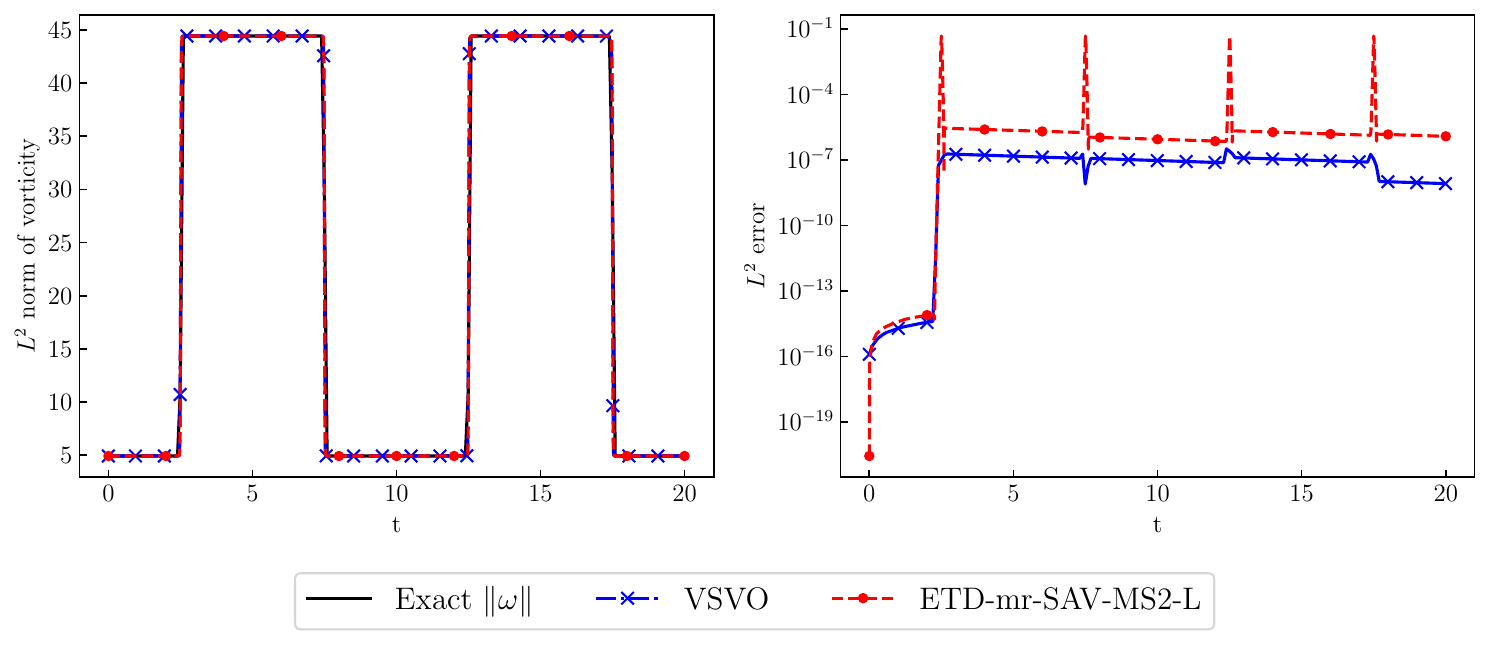}
    \caption{Evolution of the $L^2$ norm of the vorticity and the $L^2$ error computed by the VSVO method and the ETD-mr-SAV-MS2-L scheme with constant stepsizes, respectively.}
    \label{fig:VSVO_1}
\end{figure}

\begin{figure}[htbp]
    \centering
    \includegraphics[width=\linewidth]{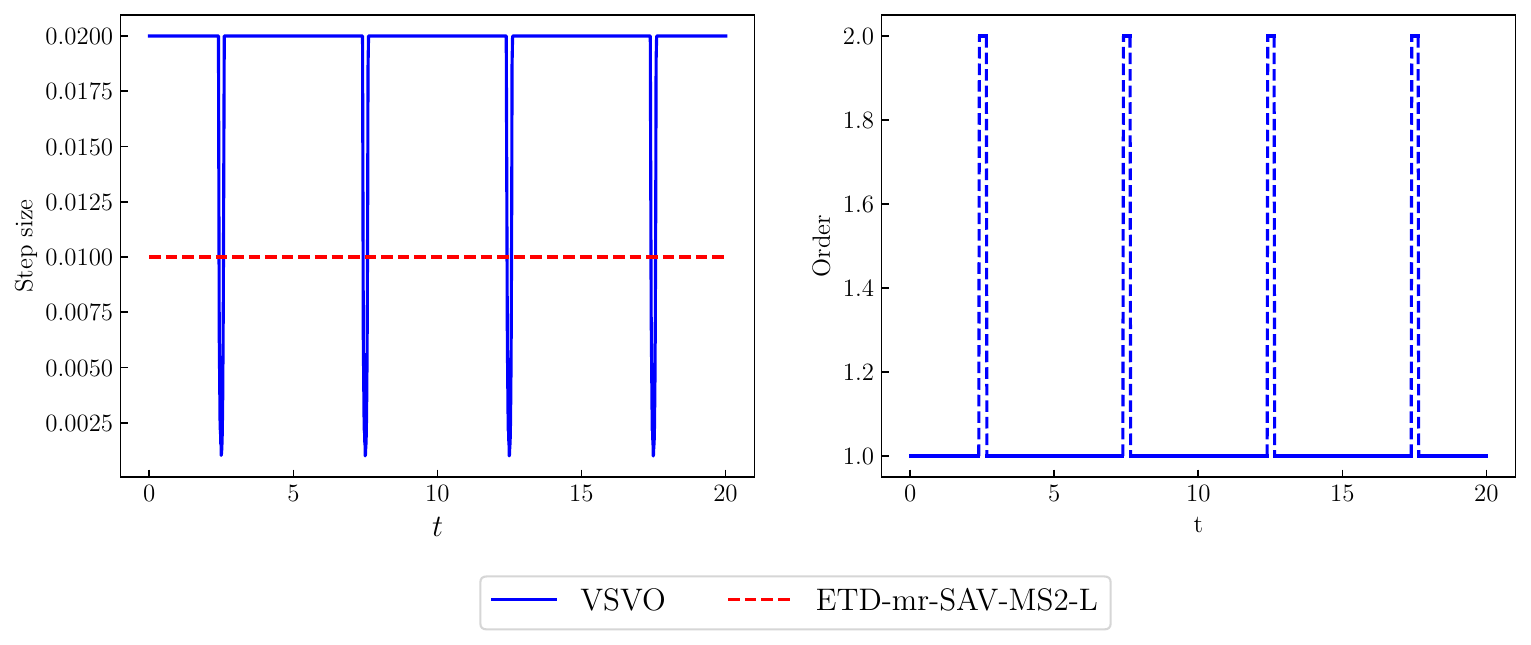}
    \caption{Adaptive stepsizes and order selection for the VSVO scheme.}
    \label{fig:VSVO_2}
\end{figure}

\section{Conclusion}\label{sec:6}

We have developed a linear, uniquely solvable, variable-step embedded exponential time-differencing (ETD) scheme for the incompressible Navier--Stokes equations in vorticity--streamfunction formulation on a two-dimensional periodic domain. The method consists of a second-order scheme together with an embedded first-order companion, providing a natural mechanism for adaptive time stepping and a posteriori error control. The scheme is long-time stable in the sense that the discrete vorticity admits a uniform-in-time bound in $L^\infty(0,\infty;L^2)$, independently of the time-step sizes and for arbitrary Reynolds numbers, provided that the external forcing is uniformly bounded in time in $L^2$. The construction combines the ETD framework with a mean-reverting scalar auxiliary variable (mr-SAV) formulation, whose mean-reverting mechanism is essential for recovering the dissipative structure needed for long-time stability.

A key feature of the present method is that each time step involves only linear subproblems. More precisely, in the periodic setting the update reduces to the evaluation of two time-dependent heat equations together with a linear equation for the scalar auxiliary variable. This fully linear structure distinguishes the present approach from earlier ETD--mr-SAV formulations that require the solution of a nonlinear algebraic equation at each step, while still retaining the desirable long-time stability property. The numerical results confirm the predicted second-order temporal accuracy, the uniform-in-time stability of the method, and the effectiveness of the embedded adaptive strategy in adjusting the time-step size to the evolving dynamics.

We have also included preliminary numerical tests of a variable-order/variable-step (VOVS/VSVO) approach based on a variable-step BDF2 discretization for the Navier--Stokes equations. Although this direction is not the main focus of the present paper, the results suggest that it may provide another viable adaptive strategy for simulations that involve rapid temporal transition.

Several natural directions remain for future work. On the analytical side, it would be important to establish a rigorous convergence analysis for the proposed linear embedded ETD scheme, including error estimates under variable time stepping, and to investigate the convergence of long-time statistical quantities generated by the numerical dynamics. On the algorithmic side, an interesting next step is the development of higher-order embedded pairs or triplets that simultaneously retain uniform-in-time bounds, variable-step capability, and low per-step complexity. Such schemes would provide a broader framework for testing and refining adaptive strategies, including variable-order/variable-step approaches. Other directions include extending the methodology beyond the periodic setting, and applying the linear ETD--mr-SAV framework to more complex dissipative fluid models, such as the Boussinesq equations and coupled multiphase flow systems.

\section*{Acknowledgment}
The second author acknowledges helpful conversation with Professors Wenbin Chen, Qiang Du, and Jie Shen.

\section*{Declarations}

\textbf{Funding} 
This work was supported in part by NSFC12271237.

\textbf{Conflict of interest} 
The authors declare that they have no known competing financial interests or personal relationships that could have appeared to influence the work reported in this paper.

\textbf{Data availability statement} 
The data and codes generated during the current study are available from the corresponding author upon reasonable request.

\bibstyle{wileyNJD-Chicago}
\bibliography{references}

\end{document}